\documentclass[reqno]{amsart}

\usepackage{amsmath}
\usepackage{amssymb}
\usepackage{graphicx}
\usepackage{adjustbox}
\usepackage[dvipsnames]{xcolor}
\usepackage{hyperref}
\hypersetup{hidelinks}
\definecolor{revisionblue}{RGB}{0,0,0}
\definecolor{strictgreen}{RGB}{0,0,0}
\definecolor{branchblue}{RGB}{0,0,0}
\definecolor{faceorange}{RGB}{0,0,0}
\definecolor{latestbrown}{RGB}{0,0,0}
\definecolor{currentbrown}{RGB}{0,0,0}
\definecolor{finalbrown}{RGB}{0,0,0}
\definecolor{auditbrown}{RGB}{0,0,0}
\definecolor{lastbrown}{RGB}{0,0,0}
\definecolor{thisbrown}{RGB}{0,0,0}
\definecolor{todaybrown}{RGB}{0,0,0}
\definecolor{revisionpurple}{RGB}{0,0,0}
\definecolor{rthreegreen}{RGB}{0,0,0}

\newcommand{\bluechange}[1]{#1}
\newcommand{\latestchange}[1]{#1}
\newcommand{\currentchange}[1]{#1}
\newcommand{\finalchange}[1]{#1}

\newcommand{\newchange}[1]{#1}
\newcommand{\thischange}[1]{{\color{thisbrown}#1}}
\newcommand{\todaychange}[1]{{\color{todaybrown}#1}}
\newcommand{\purplechange}[1]{{\color{revisionpurple}#1}}
\newcommand{\rthreechange}[1]{{\color{rthreegreen}#1}}
\newcommand{\stricttext}[1]{{\color{strictgreen}#1}}
\newcommand{\branchtext}[1]{{\color{branchblue}#1}}
\newcommand{\facetext}[1]{{\color{faceorange}#1}}

\newtheorem{theorem}{Theorem}[section]
\newtheorem{proposition}[theorem]{Proposition}
\newtheorem{lemma}[theorem]{Lemma}
\newtheorem{corollary}[theorem]{Corollary}
\newtheorem{remark}[theorem]{Remark}
\newtheorem*{propositionstar}{Proposition}

\newcommand{\ind}{\operatorname{ind}}
\newcommand{\rot}{\operatorname{rot}}

\title{Homological lifts of Arnold invariants $J^-$ and $J^+$}
\author{Noboru Ito}
\address{Department of Mathematics, Faculty of Engineering, Shinshu University,
Wakasato 4-17-1, Nagano, Nagano 380-8553, Japan}
\email{nito@shinshu-u.ac.jp}
\date{September 10, 2026}

\begin{document}
\begin{abstract}
Viro's Euler-integral polynomial \(P_C(q)\) and the Lanzat--Polyak
quantized-curvature polynomial \(I_q(C)\) refine Arnold's invariants
\(J^-\) and \(J^+\) for generic immersed one-component plane curves.  We
construct homological lifts of both.  The bigraded \emph{region homology}
retains the singular homology of every connected Alexander-index region;
its graded Euler characteristic is \(P_C(q)\).  The triply graded
\emph{smoothing-circle homology} is generated by the oriented circles of
the orientation-preserving smoothing and decategorifies to the smoothing
term in \(I_q(C)\).  Keeping the actual region summands and the boundary
regions of every smoothing circle gives a homological refinement of the
oriented smoothing configuration, or Seifert state.
\stricttext{An infinite family proves strictness: both polynomial data
and the ordinary homological lifts agree, while the component-graded
region homology and the branch-decomposed circle homology distinguish
every pair.}  Further constructions
recover the full \(I_q(C)\) by a vertex complex, realize the local change of
its curvature integral by edge homology, and give a canonical two-state
homology for unoriented curves.  \facetext{Viro described his Euler-integral
formula as an analogue of face state-sum formulas for quantum knot
polynomials.  Through the categorifications developed here, we obtain one
concrete homological face-state-sum model realizing that analogy.}
\end{abstract}

\maketitle

\section{Introduction}

Arnold introduced \(J^-\), \(J^+\), and \(St\) through their jumps under
the basic perestroikas of generic immersed plane curves
\cite{Arnold1994Book}.  Viro encoded \(J^-\) in an Euler-integral polynomial
and described his formula as an analogue of face state-sum formulas for
quantum knot polynomials \cite[Section~5.3]{Viro1996}.  Lanzat and Polyak
encoded \(J^+\) in a quantized-curvature polynomial
\cite{LanzatPolyak2013}; a parallel quantization of \(St\) was constructed
in \cite{Ito2023}.  The step taken here is different: we lift the first two
polynomial constructions to homology before their connected-component data
are summed.

Let \(C\colon S^1\to\mathbb R^2\) be a generic immersion with transverse
double points, equipped with an orientation \(o\), and let \(V(C)\) denote
its double-point set.  Smoothing every double point compatibly with \(o\)
gives a disjoint union of oriented circles \(\widetilde C_o\), the oriented
smoothing configuration or Seifert state.  The Alexander index is the
winding number on each complementary region, and \(R_j(C,o)\) is the union
of the regions with index \(j\).  Viro's Euler-integral polynomial is
\[
P_C(q)=\sum_j\chi(R_j(C,o))q^j\in\mathbb Z[q,q^{-1}],
\qquad
J^{-}(C)=1-\left.\left(q\frac{d}{dq}\right)^2P_C(q)\right|_{q=1}.
\]
Here \(\chi\) denotes Euler characteristic.  Lanzat and Polyak's
quantized-curvature polynomial has the smoothing formula
\[
I_q(C)=I_q(\widetilde C_o)
-\frac12\sum_{v\in V(C)}q^{\ind_o(v)}
\left(q^{1/2}-q^{-1/2}\right),
\]
where
\[
I_q(\widetilde C_o)
=\sum_{K\in\pi_0(\widetilde C_o)}
\rot_o(K)q^{\ind_o(K)}.
\]
The degree \(\ind_o(v)\) is the average of the indices of the four regions
at \(v\); for a smoothing circle \(K\), \(\ind_o(K)\) is the average of
the indices of its two adjacent regions, and
\(\rot_o(K)\in\{1,-1\}\) is its rotation sign.

Our first lift is the bigraded \emph{region homology}, the singular homology
of the index level sets with its canonical decomposition
\[
H_i(R_j(C,o);\mathbb Z)
\cong
\bigoplus_{R\in\pi_0(R_j(C,o))}H_i(R;\mathbb Z).
\]
Its summands are canonically indexed by the actual connected regions.  The
triply graded \emph{smoothing-circle homology} is
generated by the oriented fundamental classes of the circles in
\(\widetilde C_o\), graded by homological degree, averaged index, and
rotation sign.  For each circle generator we also retain the two
complementary regions that contain it in their boundary.  Thus the
homologies remember which circles bound the same region and refine the
Seifert state before its connected configurations are collapsed to
numerical data.

The lifts are strict.  We construct infinitely many pairs having the same
Viro polynomial, the same Lanzat--Polyak smoothing term, the same numbers of
double points and smoothing circles, and the same number of positively
oriented smoothing circles.  At the decisive index level, even the total
\(H_0\)- and \(H_1\)-groups agree, but the first homology is distributed
differently among the connected regions.  The associated circle
configurations also differ in which circles bound the same region.

We also realize the double-point correction by a zero-differential vertex
complex, construct an edge homology whose canonical classes give the local
variation of the quantized-curvature integral, and combine the two
orientations of an unoriented curve into a canonical two-state homology.
For a multigraded free abelian group, \(\operatorname{grank}_{t,q,\ldots}\)
denotes its graded-rank polynomial, with one variable for each displayed
degree.  The main statements follow.  Section~\ref{sec:classical-formulas}
fixes the classical inputs, Sections~\ref{sec:homological-refinements}--
\ref{sec:discussion} present the constructions and their mathematical
meaning without interrupting that progression, and
\bluechange{Section~\ref{sec:proofs} gives the complete proofs.}

\begin{theorem}[Homological lifts of the Viro and
Lanzat--Polyak quantizations]
\label{thm:main}
Let \(C\) be a generic immersed one-component plane curve with orientation
\(o\).

\begin{enumerate}
\renewcommand{\labelenumi}{(\Alph{enumi})}
\renewcommand{\theenumi}{\Alph{enumi}}

\item
\label{thm:main-region}
For \(j\in\mathbb Z\), the groups
\[
\mathcal H_{\mathrm{region}}^{i,j}(C,o)
:=H_i(R_j(C,o);\mathbb Z)
\]
form a bigraded invariant under direct self-tangency and weak triple-point
perestroikas.  Retaining the decomposition by the actual regions gives
\begin{equation}
\mathcal H_{\mathrm{region}}^{i,j}(C,o)
\cong
\bigoplus_{R\in\pi_0(R_j(C,o))}H_i(R;\mathbb Z).
\label{eq:main-component-decomposition}
\end{equation}
Its graded Euler characteristic is
\begin{equation}
\operatorname{grank}_{t,q}\mathcal H_{\mathrm{region}}(C,o)
\big|_{t=-1}=P_C(q),
\label{eq:region-specialization}
\end{equation}
\purplechange{Moreover, for every \(n\geq0\),}
\begin{equation*}
{\color{revisionpurple}
\left.
D_q^n
\left(
\operatorname{grank}_{t,q}\mathcal H_{\mathrm{region}}(C,o)
\big|_{t=-1}
\right)
\right|_{q=1}
=
\sum_{i,j}(-1)^ij^n
\operatorname{rank}\mathcal H_{\mathrm{region}}^{i,j}(C,o).
\tag{8}}
\end{equation*}
\purplechange{In particular, the case \(n=2\) gives}
\[
J^{-}(C)=1-\left.
\left(q\frac{d}{dq}\right)^2
\left(\operatorname{grank}_{t,q}\mathcal H_{\mathrm{region}}(C,o)
\big|_{t=-1}\right)\right|_{q=1}.
\]

\item
\label{thm:main-circle}
For a smoothing circle \(K\subset\widetilde C_o\), put
\[
\varepsilon(K)=\frac{1+\rot_o(K)}2\in\{0,1\}.
\]
The oriented fundamental classes define the triply graded invariant
\[
\mathcal H_{\mathrm{circle}}^{i,j,\varepsilon}(C,o)
=
\begin{cases}
\displaystyle
\bigoplus_{\substack{K\in\pi_0(\widetilde C_o)\\
\ind_o(K)=j,\ \varepsilon(K)=\varepsilon}}
\mathbb Z\langle[K]\rangle,&i=1,\\[3mm]
0,&i\ne1.
\end{cases}
\]
Its graded Euler characteristic is
\begin{equation}
\operatorname{grank}_{t,q,u}\mathcal H_{\mathrm{circle}}(C,o)
\big|_{t=u=-1}=I_q(\widetilde C_o),
\label{eq:circle-specialization-detail}
\end{equation}
\purplechange{and, for every \(n\geq0\),}
\begin{equation*}
{\color{revisionpurple}
\left.
D_q^n
\left(
\operatorname{grank}_{t,q,u}\mathcal H_{\mathrm{circle}}(C,o)
\big|_{t=u=-1}
\right)
\right|_{q=1}
=
\sum_{K\in\pi_0(\widetilde C_o)}
\rot_o(K)\bigl(\ind_o(K)\bigr)^n.
\tag{9}}
\end{equation*}
\purplechange{In particular,}
\begin{equation}
J^{-}(C)=\newchange{1-2\left.\frac{d}{dq}I_q(\widetilde C_o)\right|_{q=1}
=}1-2\left.\frac{d}{dq}
\left(\operatorname{grank}_{t,q,u}\mathcal H_{\mathrm{circle}}(C,o)
\big|_{t=u=-1}\right)\right|_{q=1}.
\label{eq:main-circle-Jminus}
\end{equation}

\item
\label{thm:main-components}
\rthreechange{Direct self-tangency and weak triple-point perestroikas induce an ambient isotopy of}
\(\widetilde C_o\) that identifies
all complementary regions and all smoothing circles.  It preserves
decomposition~\eqref{eq:main-component-decomposition} and, for every circle
\(K\), the two actual regions \(R\) satisfying \(K\subset\partial R\).
\branchtext{Let \(T(\widetilde C_o)\) be the dual tree of the circle
configuration, with root the vertex \(v_\infty\) corresponding to the
unbounded region.  We call each connected component of
\(T(\widetilde C_o)\setminus\{v_\infty\}\) a branch.  Assigning each
smoothing circle to the branch containing the non-root endpoint of its
dual edge partitions the circle generators and gives a branch-indexed
family of smoothing-circle homologies, well defined up to permutation of
the branches.  This branch
decomposition is preserved by these perestroikas.}
\purplechange{Writing \(\beta_o(K)\) for the assigned branch, set}
\[
{\color{revisionpurple}
\mathcal H_{\mathrm{circle}}^{i,j,\varepsilon}(C,o;B)
=
\begin{cases}
\displaystyle
\bigoplus_{\substack{K\in\pi_0(\widetilde C_o), \beta_o(K)=B\\
\ind_o(K)=j, \varepsilon(K)=\varepsilon}}
\mathbb Z\langle[K]\rangle,&i=1,\\[3mm]
0,&i\ne1.
\end{cases}}
\]
\purplechange{The branch-indexed invariant and its numerical
Poincar\'e evaluation are, respectively,}
\begin{equation*}
{\color{revisionpurple}
\left(\mathcal H_{\mathrm{circle}}^{i,j,\varepsilon}(C,o;B)\right)_
{B\in\pi_0(T(\widetilde C_o)\setminus\{v_\infty\})}.}
\tag{13}
\end{equation*}
\begin{equation*}
{\color{revisionpurple}\adjustbox{max width=0.90\linewidth}{$\displaystyle
\mathcal P_{\mathrm{branch}}(C,o;t,q,u,w):=
\sum_{B\in\pi_0(T(\widetilde C_o)\setminus\{v_\infty\})}
\left(\sum_{\substack{K\in\pi_0(\widetilde C_o)\\\beta_o(K)=B}}
tq^{\ind_o(K)}u^{\varepsilon(K)}\right)
w^{\#\beta_o^{-1}(B)}.$}}
\tag{14}
\end{equation*}

\rthreechange{Moreover, there exists a homological face-state-sum model in which the vertices of the rooted dual tree are the faces, the Alexander index is their coloring, \(H_*(R;\mathbb Z)\) is the homological face weight, and the circle generators lie on the dual edges (Section~\ref{sec:discussion}).}

{\color{lastbrown}
More precisely, for an orientation state \(s\in\{o,-o\}\), a face state is
a choice of an actual marked face
\(R_*\in\pi_0(\mathbb R^2\setminus\widetilde C_s)\).  Give each face \(R\)
the local weight
\[
W_{R_*}(R)=
\begin{cases}
\displaystyle
\left(\sum_i\operatorname{rank}H_i(R;\mathbb Z)t^i\right)
q^{\ind_s(R)}[R],&R=R_*,\\[3mm]
1,&R\ne R_*.
\end{cases}
\]
Then the concrete homological face-state sum is
\begin{equation*}
\begin{aligned}
\mathcal Z_{\mathrm{face}}(C,s;t,q)
&:=\sum_{R_*\in\pi_0(\mathbb R^2\setminus\widetilde C_s)}
\prod_R W_{R_*}(R)\\
&=\sum_j\sum_{R\in\pi_0(R_j(C,s))}\sum_i
\operatorname{rank}H_i(R;\mathbb Z)t^iq^j[R].
\end{aligned}
\tag{\color{revisionpurple}15}
\end{equation*}
\thischange{The specialization \([R]\mapsto1\), followed by \(t=-1\),
gives \(P_C(q)\) for \(s=o\) and \(P_C(q^{-1})\) for \(s=-o\).}
\purplechange{For an unoriented curve, the two orientation states give}
\begin{equation*}
{\color{revisionpurple}
\mathcal Z_{\mathrm{face}}^{\mathrm{un}}(C;t,q)
:=\sum_{s\in\{o,-o\}}\sum_j
\sum_{R\in\pi_0(R_j(C,s))}\sum_i
\operatorname{rank}H_i(R;\mathbb Z)t^iq^j[s,R].}
\tag{16}
\end{equation*}
\purplechange{Equations (15) and (16) are invariant under direct
self-tangency and weak triple-point perestroikas, and (16) is independent
of the temporary orientation.}
}

\item
\label{thm:main-comparison}
The two decategorifications satisfy
\begin{equation}
\operatorname{grank}_{t,q}\mathcal H_{\mathrm{region}}(C,o)\big|_{t=-1}
=1+\left(q^{1/2}-q^{-1/2}\right)
\operatorname{grank}_{t,q,u}\mathcal H_{\mathrm{circle}}(C,o)
\big|_{t=u=-1},
\label{eq:main-comparison}
\end{equation}
which is the classical identity
\begin{equation}
P_C(q)=1+\left(q^{1/2}-q^{-1/2}\right)I_q(\widetilde C_o).
\label{eq:main-comparison-classical}
\end{equation}

\item
\label{thm:main-double-point}
The zero-differential vertex complex \(\mathcal V(C,o)\), with generators
at each \(v\in V(C)\) in bidegrees
\((0,\ind_o(v)+\tfrac12)\) and
\((1,\ind_o(v)-\tfrac12)\), satisfies
\begin{equation}
I_q(C)=
\operatorname{grank}_{t,q,u}\mathcal H_{\mathrm{circle}}(C,o)
\big|_{t=u=-1}
-\frac12\operatorname{grank}_{t,q}\mathcal V(C,o)\big|_{t=-1}.
\label{eq:main-full-LP}
\end{equation}
Consequently
\begin{equation}
J^{+}(C)=1-2\left.\frac{d}{dq}I_q(C)\right|_{q=1}.
\tag{\todaychange{7'}}
\label{eq:main-full-LP-Jplus}
\end{equation}

\item
\label{thm:main-edge}
Write \(e_l^v,e_r^v\) for the two incoming edges at a double point \(v\).
The two-term incidence complex generated by these edges has
\[
H_1\cong\bigoplus_{v\in V(C)}\mathbb Z
\langle[e_l^v-e_r^v]\rangle,
\qquad H_0=0.
\]
If \(\beta_v\) denotes the magnitude of the pair of opposite turns created
by smoothing \(v\), the canonical class at \(v\) has
weighted signed \(q\)-character
\[
\frac{\beta_v}{2\pi}
\left(q^{\ind_o(v)+1/2}-q^{\ind_o(v)-1/2}\right),
\]
equal to the local variation of the normalized quantized-curvature integral.
\purplechange{The total change produced by smoothing all double points is}
\[
{\color{revisionpurple}
\frac{1}{2\pi}
\sum_{v\in V(C)}
\beta_v
\left(q^{\ind_o(v)+1/2}-q^{\ind_o(v)-1/2}\right).}
\]

\item
\label{thm:main-unoriented}
\purplechange{For an unoriented \(C\), the two-state homology}
\begin{equation*}
{\color{revisionpurple}
\mathcal H_{\mathrm{circle}}^{\mathrm{un}}(C)
=\mathcal H_{\mathrm{circle}}(C,o)
\oplus\mathcal H_{\mathrm{circle}}(C,-o)
\tag{11}}
\end{equation*}
\purplechange{is independent of the temporary choice of orientation and is
invariant under direct self-tangency and weak triple-point perestroikas.  Its
canonical orientation-exchange involution sends degree
\((i,j,\varepsilon)\) to \((i,-j,1-\varepsilon)\).  Put}
\[
{\color{revisionpurple}
\begin{aligned}
G_o(q)&:=
\left.\operatorname{grank}_{t,q,u}
\mathcal H_{\mathrm{circle}}(C,o)\right|_{t=u=-1},\\
G_{\mathrm{un}}(q)&:=
\left.\operatorname{grank}_{t,q,u}
\mathcal H_{\mathrm{circle}}^{\mathrm{un}}(C)\right|_{t=u=-1}.
\end{aligned}}
\]
\purplechange{Then}
\[
{\color{revisionpurple}
G_{\mathrm{un}}(q)=G_o(q)-G_o(q^{-1}),}
\]
\purplechange{and, for every \(m\geq0\),}
\begin{equation*}
{\color{revisionpurple}
\left.D_q^mG_{\mathrm{un}}(q)\right|_{q=1}
=\bigl(1-(-1)^m\bigr)
\left.D_q^mG_o(q)\right|_{q=1}.
\tag{19}}
\end{equation*}
\purplechange{Thus the odd moments double and the even moments cancel.  In
particular,}
\begin{equation*}
{\color{revisionpurple}
J^-(C)=1-G_{\mathrm{un}}'(1),
\tag{20}}
\end{equation*}
\purplechange{and}
\begin{equation*}
{\color{revisionpurple}
J^+(C)=1-G_{\mathrm{un}}'(1)+\#V(C).
\tag{21}}
\end{equation*}

\item
\label{thm:main-infinite-pairs}
{\color{strictgreen}
There are infinitely many pairs \((C_n,o_n)\), \((C'_n,o'_n)\), \(n\ge1\),
with equal double-point counts, equal smoothing-circle counts, equal numbers
of positively oriented smoothing circles, and
\[
I_q(\widetilde{(C_n)}_{o_n})=I_q(\widetilde{(C'_n)}_{o'_n}),
\qquad
P_{C_n}(q)=P_{C'_n}(q),
\]
\purplechange{Their ordinary bigraded region homologies are isomorphic:}
\[
{\color{revisionpurple}
\mathcal H_{\mathrm{region}}^{*,*}(C_n,o_n)
\cong
\mathcal H_{\mathrm{region}}^{*,*}(C'_n,o'_n).}
\]
\purplechange{By contrast,} the component grading introduced in
Section~\ref{sec:component-refinements} gives triply graded groups
\[
\mathcal H_{\mathrm{region}}^{i,j,k}(C,o)
:=
\bigoplus_{\substack{R\in\pi_0(R_j(C,o))\\
\#\pi_0(\partial R)-1=k}}
H_i(R;\mathbb Z)
\]
that are not isomorphic for the two members of any pair:
\[
\mathcal H_{\mathrm{region}}^{*,*,*}(C_n,o_n)
\not\cong
\mathcal H_{\mathrm{region}}^{*,*,*}(C'_n,o'_n).
\]
{\color{revisionpurple}
Moreover, the coefficient of \(q^1\) in their component Poincar\'e
polynomials is, respectively,
\[
(n+1)(1+2t)v^2
\qquad\text{and}\qquad
\bigl(1+(2n+2)t\bigr)v^{2n+2}+n.
\]
These coefficients are unequal, and hence
\[
\mathcal P_{\mathrm{region}}(C_n,o_n;t,q,v)
\ne
\mathcal P_{\mathrm{region}}(C'_n,o'_n;t,q,v).
\]
}
\rthreechange{The ordinary triply graded smoothing-circle homologies are isomorphic:}
\[
\rthreechange{\mathcal H_{\mathrm{circle}}^{*,*,*}(C_n,o_n)
\cong
\mathcal H_{\mathrm{circle}}^{*,*,*}(C'_n,o'_n).}
\]
Their branch-indexed families, however, do not agree.  With the branches defined by
removing the root of the dual tree as in
Section~\ref{sec:component-refinements}, one has
\[
\begin{aligned}
&\left(
\mathcal H_{\mathrm{circle}}^{*,*,*}(C_n,o_n;B)
\right)_{B\in\pi_0(T(\widetilde{(C_n)}_{o_n})\setminus\{v_\infty\})}
\\[-1mm]
&\qquad\not\cong
\left(
\mathcal H_{\mathrm{circle}}^{*,*,*}(C'_n,o'_n;B')
\right)_{B'\in\pi_0(T(\widetilde{(C'_n)}_{o'_n})\setminus\{v_\infty\})}.
\end{aligned}
\]
where isomorphism of the displayed families allows a permutation of the
branches but must preserve all three gradings in every corresponding factor.
\purplechange{Their branch Poincar\'e polynomials likewise satisfy}
\[
{\color{revisionpurple}
\mathcal P_{\mathrm{branch}}(C_n,o_n;t,q,u,w)
\ne
\mathcal P_{\mathrm{branch}}(C'_n,o'_n;t,q,u,w).}
\]
Thus Viro's polynomial admits a natural bigraded homological lift which
still agrees on this family, while the component-graded region homology and
the branch-decomposed smoothing-circle homology, together with their
displayed Poincar\'e evaluations, distinguish every pair.
}
\end{enumerate}
\end{theorem}

\begin{remark}
\rthreechange{In the Lanzat--Polyak normalization \cite{LanzatPolyak2013}, if \(\theta_v\in(0,\pi)\) is the non-oriented angle between the two branches at \(v\), then \(\beta_v=\pi-\theta_v\).  Hence the total change in Part~{\rm (F)} is
\[
\frac{1}{2\pi}
\sum_{v\in V(C)}(\pi-\theta_v)q^{\ind_o(v)}
\left(q^{1/2}-q^{-1/2}\right).
\]
Thus the classical Lanzat--Polyak expression is recovered from the preceding formula by this change of variables.}
\end{remark}

\section{Classical formulas for Arnold's \texorpdfstring{$J^{-}$}{J-} and quantized
curvature invariants}
\label{sec:classical-formulas}

We recall the classical formulas in the common notation used by the
homological constructions.

Let
\[
C\colon S^1\longrightarrow\mathbb R^2
\]
be an oriented generic one-component plane curve, and let $V(C)$
denote its set of double points.  A self-tangency perestroika is
called positive if the number of double points increases.

Arnold's invariant $J^{-}$ is unchanged under direct
self-tangency perestroikas and triple-point perestroikas, and
decreases by $2$ under a positive inverse self-tangency perestroika
\cite{Arnold1994Book,Viro1996}.  We use Arnold's normalization as stated
by Viro: it is the normalization for which the identity
\eqref{eq:Viro-Jminus} below holds, and Viro verifies that identity on
Arnold's standard curves \cite[Subsection~3.1]{Viro1996}.

Let $\widetilde C_o$ be the orientation-preserving smoothing of all
double points of $C$, and let
\[
\ind_{\widetilde C_o}\colon
\mathbb R^2\setminus\widetilde C_o
\longrightarrow
\mathbb Z
\]
be its Alexander index, namely the mapping degree, equivalently the
winding number, of $\widetilde C_o$ about a point of its complement.  The integer values
of this function on the complementary regions form the Alexander
numbering.  Throughout, the term \emph{Alexander index} refers either
to this locally constant function or to one of its values.  For each
$j\in\mathbb Z$, set
\[
R_j(C,o)
=
\left\{
x\in\mathbb R^2\setminus\widetilde C_o
\mathrel{}\middle|\mathrel{}
\ind_{\widetilde C_o}(x)=j
\right\}.
\]
Thus $R_j(C,o)$ is the full level set on which the Alexander index
equals $j$, and it may have more than one connected component.  Throughout,
a \emph{region} means a connected component of
\(\mathbb R^2\setminus\widetilde C_o\).

Viro defined the Laurent polynomial
\cite{Viro1996}
\[
P_C(q)
=
\int_{\mathbb R^2\setminus\widetilde C_o}
q^{\ind_{\widetilde C_o}(x)}\,d\chi(x).
\]
This is a Laurent polynomial in \(\mathbb Z[q,q^{-1}]\); equivalently,
\begin{equation*}
P_C(q)
=
\sum_{j\in\mathbb Z}
\chi\bigl(R_j(C,o)\bigr)q^j.
\tag{I1}\label{eq:Viro-polynomial}
\end{equation*}
Here \(\chi\) is Euler characteristic and \(d\chi\) denotes Euler
integration.
The unbounded complementary component is included in the appropriate
level set $R_j(C,o)$.

Put
\[
D_q:=q\frac{d}{dq}.
\]
Then \(D_q^n(q^j)=j^nq^j\), while
\((D_qf)(1)=f'(1)\).

Following Viro's terminology
\cite[Subsection~5.1]{Viro1996}, define
\[
M_r(C)
:=
\int_{\mathbb R^2\setminus\widetilde C_o}
\bigl(\ind_{\widetilde C_o}(x)\bigr)^r\,d\chi(x)
\]
and call \(M_r(C)\) the \(r\)-th \emph{index moment}.
Since
\[
D_q^r(q^j)=j^rq^j,
\]
we have
\[
\begin{aligned}
\left.
D_q^rP_C(q)
\right|_{q=1}
&=
\sum_{j\in\mathbb Z}
j^r\chi\bigl(R_j(C,o)\bigr)\\
&=
M_r(C).
\end{aligned}
\]
In particular, writing \(\rot(C)\) for the rotation number of \(C\),
\[
\rot(C)
=
\left.D_qP_C(q)\right|_{q=1}
\]
and
\begin{equation*}
J^{-}(C)
=
1-
\left.D_q^2P_C(q)\right|_{q=1}.
\tag{I2}\label{eq:Viro-Jminus}
\end{equation*}

We next recall the quantized-curvature invariant of Lanzat and Polyak
\cite{LanzatPolyak2013}.  Let \(\ind_C\) denote the extension of the
Alexander index to a regular point of $C$ obtained by averaging the indices
of the two adjacent regions, and to a double point by averaging those of the
four adjacent regions.  For each double point
\[
v=C(t_1)=C(t_2)\in V(C),
\]
let $\theta_v\in(0,\pi)$ be the non-oriented angle between
$C'(t_1)$ and $-C'(t_2)$.  Lanzat and Polyak defined
\[
I_q(C)
=
\frac{1}{2\pi}
\left(
\int_{S^1}
\kappa(t)\,
q^{\ind_C(C(t))}\,dt
-
\sum_{v\in V(C)}
\theta_v\,
q^{\ind_C(v)}
\left(q^{1/2}-q^{-1/2}\right)
\right).
\]
Thus \(I_q(C)\in\mathbb R[q^{1/2},q^{-1/2}]\).
Here \(\kappa(t)\) is the signed curvature of the parametrized curve.
We write \(\ind_o\) in place of \(\ind_C\) when emphasizing the chosen
orientation.

Write
\[
\widetilde C_o
=
\coprod_{\latestchange{\ell}} \widetilde C_{o,\latestchange{\ell}}.
\]
The averaged Alexander index is constant on each $\widetilde C_{o,\latestchange{\ell}}$.  Hence, by Hopf's Umlaufsatz \cite{Hopf1935},
\begin{equation*}
I_q(\widetilde C_o)
=
\sum_{\latestchange{\ell}}
\rot_o(\widetilde C_{o,\latestchange{\ell}})
q^{\ind_o(\widetilde C_{o,\latestchange{\ell}})}.
\tag{I4}\label{eq:LP-smoothed}
\end{equation*}
Moreover,
\begin{equation*}
I_q(C)
=
I_q(\widetilde C_o)
-
\frac12
\sum_{v\in V(C)}
q^{\ind_o(v)}
\left(q^{1/2}-q^{-1/2}\right).
\tag{I3}\label{eq:LP-smoothing-relation}
\end{equation*}

The Lanzat--Polyak comparison identity used below is
\begin{equation*}
P_C(q)=1+\left(q^{1/2}-q^{-1/2}\right)I_q(C)
+\frac12\sum_{v\in V(C)}q^{\ind_o(v)}
\left(q^{1/2}-q^{-1/2}\right)^2.
\tag{I5}\label{eq:LP-comparison}
\end{equation*}

Orientation reversal gives
\[
P_{-C}(q)=P_C(q^{-1})
\]
and
\[
I_q(-C)=-I_{q^{-1}}(C)
\]
\cite{LanzatPolyak2013}.  In particular,
\[
I_q(\widetilde C_{-o})
=
-I_{q^{-1}}(\widetilde C_o).
\]
Thus the quantity obtained from $P_C(q)$ by applying $D_q^2$ and
evaluating at $q=1$ is unchanged by orientation reversal.  For the
Lanzat--Polyak smoothing term, the quantities obtained by applying
$D_q^n$ at $q=1$ are preserved for odd $n$ and change sign for even
$n$.

\section{Homological refinements of the Viro and
Lanzat--Polyak quantizations}
\label{sec:homological-refinements}

Fix an orientation \(o\) of \(C\).  We first construct the two principal
homological lifts, then prove their invariance, including the decomposition
by actual regions, in a single argument.

\subsection{Region homology}
\label{subsec:region}

For each \(j\in\mathbb Z\), the index level \(R_j(C,o)\) is a finite
disjoint union of connected planar regions.  Its singular homology gives
the bigraded group in Part~{\rm (\ref{thm:main-region})} of
Theorem~\ref{thm:main}.  We call it the \emph{region homology} of
\((C,o)\).

\begin{proposition}[Region direct-sum decomposition and decategorification]
\label{prop:region-decategorification}
\purplechange{The region homology has the canonical decomposition}
\begin{equation*}
{\color{revisionpurple}
\mathcal H_{\mathrm{region}}^{i,j}(C,o)
\cong
\bigoplus_{R\in\pi_0(R_j(C,o))}H_i(R;\mathbb Z),
\tag{1}}
\end{equation*}
\purplechange{and its graded Euler characteristic is}
\begin{equation*}
{\color{revisionpurple}
\operatorname{grank}_{t,q}\mathcal H_{\mathrm{region}}(C,o)
\big|_{t=-1}=P_C(q).
\tag{2}}
\end{equation*}
\purplechange{For every \(n\ge0\),}
\begin{equation}
\label{eq:region-moments}
\left.
D_q^n
\left(
\operatorname{grank}_{t,q}\mathcal H_{\mathrm{region}}(C,o)
\big|_{t=-1}
\right)
\right|_{q=1}
=
\sum_{i,j}(-1)^ij^n
\operatorname{rank}\mathcal H_{\mathrm{region}}^{i,j}(C,o).
\end{equation}
The left-hand side is Viro's \(n\)-th index moment; in particular, \(n=2\)
gives the formula for \(J^{-}\) in Theorem~\ref{thm:main}.
\end{proposition}

\bluechange{The complete verification is given in
Section~\ref{proof:region-decategorification}.}

\subsection{Smoothing-circle homology and boundary regions}
\label{subsec:circle}

Each oriented component \(K\subset\widetilde C_o\)
determines its fundamental class
\[
[K]\in H_1(K;\mathbb Z),
\]
its averaged Alexander index
\(\ind_o(K)\in\frac12\mathbb Z\), and its rotation sign
\(\rot_o(K)\in\{1,-1\}\).  Placing \([K]\) in homological degree \(1\)
and using \(\varepsilon(K)=(1+\rot_o(K))/2\) gives the group in
Part~{\rm (\ref{thm:main-circle})}.  We call it the
\emph{smoothing-circle homology} of \((C,o)\).

\begin{proposition}[Circle graded rank and sign specialization]
\label{prop:circle-decategorification}
The full graded-rank polynomial is
\begin{equation*}
\operatorname{grank}_{t,q,u}\mathcal H_{\mathrm{circle}}(C,o)
=
\sum_{K\in\pi_0(\widetilde C_o)}
tq^{\ind_o(K)}u^{\varepsilon(K)}.
\end{equation*}
\purplechange{Its sign specialization is}
\begin{equation*}
{\color{revisionpurple}
\operatorname{grank}_{t,q,u}\mathcal H_{\mathrm{circle}}(C,o)
\big|_{t=u=-1}=I_q(\widetilde C_o).
\tag{3}}
\end{equation*}
\purplechange{More generally, for every \(n\ge0\),}
\begin{equation}
\label{eq:circle-moments}
\left.
D_q^n
\left(
\operatorname{grank}_{t,q,u}\mathcal H_{\mathrm{circle}}(C,o)
\big|_{t=u=-1}
\right)
\right|_{q=1}
=
\sum_{K\in\pi_0(\widetilde C_o)}
\rot_o(K)\bigl(\ind_o(K)\bigr)^n.
\end{equation}
\end{proposition}

\bluechange{The sign check and the moment calculation are given in
Section~\ref{proof:circle-decategorification}.}

\subsection{Invariance of the homological lifts}
\label{subsec:component-invariance}

\begin{lemma}[Invariance of the smoothed configuration]
\label{lem:smoothed-configuration-invariance}
Let \((C,o)\) and \((C',o')\) be related by a direct self-tangency
perestroika or a weak triple-point perestroika, with the orientation
transported through the move.  Then there is an ambient isotopy
\[
F\colon\mathbb R^2\times[0,1]\longrightarrow\mathbb R^2,
\qquad F(x,t)=F_t(x),
\]
with \(F_0=\operatorname{id}_{\mathbb R^2}\) and
\(F_1(\widetilde C_o)=\widetilde C'_{o'}\).  It
preserves the Alexander indices of the complementary regions, the averaged
indices and rotation signs of the smoothing circles, and the boundary
relation \(K\subset\partial R\).
\end{lemma}

\bluechange{Its proof is given in
Section~\ref{proof:smoothed-configuration-invariance}.}

\begin{proposition}[Invariance of the homological lifts]
\label{prop:component-refinement-invariance}
Under the hypotheses of
Lemma~\ref{lem:smoothed-configuration-invariance}, for every \(i,j\) the
final homeomorphism \(F_1\) induces an isomorphism
\[
\bigoplus_{R\in\pi_0(R_j(C,o))}H_i(R;\mathbb Z)
\xrightarrow{\ \bigoplus_R(F_1|_R)_*\ }
\bigoplus_{R'\in\pi_0(R_j(C',o'))}H_i(R';\mathbb Z),
\]
where the summand indexed by \(R\) is sent to the summand indexed by
\(F_1(R)\).  This is the isomorphism
\[
\mathcal H_{\mathrm{region}}^{i,j}(C,o)
\cong
\mathcal H_{\mathrm{region}}^{i,j}(C',o').
\]
On the circle homology, \(F_1\) sends
\[
[K]\longmapsto[F_1(K)]
\]
and gives a trigraded isomorphism.  Moreover,
\[
K\subset\partial R
\quad\Longleftrightarrow\quad
F_1(K)\subset\partial F_1(R).
\]
\end{proposition}

\bluechange{Its proof is given in
Section~\ref{proof:homological-invariance}.}

\purplechange{Propositions~\ref{prop:region-decategorification} and
\ref{prop:component-refinement-invariance} prove
Part~{\rm (\ref{thm:main-region})} of Theorem~\ref{thm:main}.}

\subsection{Comparison of the two quantizations}
\label{subsec:comparison}

\begin{proposition}
\label{prop:comparison}
The region and smoothing-circle decategorifications satisfy
\begin{equation*}
\rthreechange{\adjustbox{max width=0.96\linewidth}{$\displaystyle
\operatorname{grank}_{t,q}\mathcal H_{\mathrm{region}}(C,o)\big|_{t=-1}
=1+\left(q^{1/2}-q^{-1/2}\right)\operatorname{grank}_{t,q,u}\mathcal H_{\mathrm{circle}}(C,o)\big|_{t=u=-1}. $}}
\tag{5}
\end{equation*}
\purplechange{Equivalently,}
\begin{equation*}
{\color{revisionpurple}
P_C(q)=1+\left(q^{1/2}-q^{-1/2}\right)I_q(\widetilde C_o).}
\tag{6}
\end{equation*}
\end{proposition}

\purplechange{Proposition~\ref{prop:comparison} proves
Part~{\rm (\ref{thm:main-comparison})} of Theorem~\ref{thm:main}.}

\begin{corollary}[The smoothing-circle formula for \(J^{-}\)]
\label{cor:circle-Jminus}
The smoothing-circle homology recovers \(J^{-}\) by
\begin{equation*}
\color{revisionpurple}
\begin{aligned}
J^{-}(C)
&=1-2\left.\frac{d}{dq}I_q(\widetilde C_o)\right|_{q=1}\\
&=1-2\left.\frac{d}{dq}
\left(
\operatorname{grank}_{t,q,u}\mathcal H_{\mathrm{circle}}(C,o)
\big|_{t=u=-1}
\right)\right|_{q=1}.
\end{aligned}
\tag{4}
\end{equation*}
\end{corollary}

\currentchange{The calculation proving this formula is given in
Section~\ref{proof:circle-Jminus}.}
\purplechange{Propositions~\ref{prop:circle-decategorification} and
\ref{prop:component-refinement-invariance}, together with
Corollary~\ref{cor:circle-Jminus}, prove
Part~{\rm (\ref{thm:main-circle})} of Theorem~\ref{thm:main}.}

\subsection{The double-point correction and the full
Lanzat--Polyak invariant}
\label{subsec:vertex}

For each \(v\in V(C)\), put \(j(v)=\ind_o(v)\).  Let
\(\mathcal V(C,o)\) be the zero-differential bigraded complex with
generators \(v_0,v_1\) in bidegrees
\[
\deg(v_0)=\left(0,j(v)+\frac12\right),
\qquad
\deg(v_1)=\left(1,j(v)-\frac12\right).
\]
Its graded-rank polynomial therefore satisfies
\begin{equation}
\label{eq:vertex-specialization}
\operatorname{grank}_{t,q}\mathcal V(C,o)\big|_{t=-1}
=\sum_{v\in V(C)}q^{j(v)}
\left(q^{1/2}-q^{-1/2}\right).
\end{equation}

\begin{proposition}
\label{prop:full-LP}
\purplechange{The circle homology and the vertex complex satisfy}
\begin{equation*}
{\color{revisionpurple}\adjustbox{max width=0.90\linewidth}{$\displaystyle
I_q(C)=\operatorname{grank}_{t,q,u}\mathcal H_{\mathrm{circle}}(C,o)\big|_{t=u=-1}
-\frac12\operatorname{grank}_{t,q}\mathcal V(C,o)\big|_{t=-1}.$}}
\tag{7}
\end{equation*}
\purplechange{Consequently,}
\begin{equation*}
{\color{revisionpurple}
J^{+}(C)=1-2\left.\frac{d}{dq}I_q(C)\right|_{q=1}.
\tag{7'}}
\end{equation*}
\end{proposition}

\purplechange{The verification of both formulas is given in
Section~\ref{proof:full-LP}.}
\purplechange{Proposition~\ref{prop:full-LP} proves
Part~{\rm (\ref{thm:main-double-point})} of
Theorem~\ref{thm:main}.}

\subsection{The edge model behind the double-point correction}
\label{subsec:edge}

The vertex complex records the double-point correction algebraically.  The
edge construction below identifies the same two local terms with the signed
bends created by orientation-preserving smoothing, giving a homological
model for the local change of the quantized-curvature integral.

Regard \(C\) as an oriented four-valent plane graph whose vertex set
is \(V(C)\) and whose edges are the closures of the connected
components of \(C\setminus V(C)\).  For each \(v\in V(C)\), exactly
two oriented edges have terminal vertex \(v\).  In the
orientation-preserving coordinates shown in
Figure~\ref{fig:incoming-edges}, denote these globally determined
incoming edges by \(e_l^v\) and \(e_r^v\).

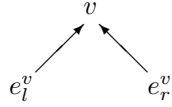
\begin{figure}[ht]
\begin{picture}(50,30)
\put(-2,0){$e^v_{l}$}
\put(50,0){$e^v_{r}$}
\put(26,30){$v$}
\put(8,8){\vector(1,1){18}}
\put(50,8){\vector(-1,1){18}}
\end{picture}
\caption{The two incoming edges at a double point.}
\label{fig:incoming-edges}
\end{figure}

Let
\[
E(C,o)
=
\coprod_{v\in V(C)}
\{e_l^v,e_r^v\}.
\]
When \(V(C)=\varnothing\), we set \(E(C,o)=\varnothing\).

Each incoming edge is adjacent to two complementary regions.  Its averaged Alexander index is the average of the Alexander indices
of those two regions.  Similarly, the averaged Alexander index of a
double point is the average of the Alexander indices of its four
adjacent regions \cite{Shumakovitch1996}.  With the local convention of
Figure~\ref{fig:incoming-edges},
\[
\ind_o(e_l^v)
=
j(v)+\frac12,
\qquad
\ind_o(e_r^v)
=
j(v)-\frac12.
\]

Let
\[
\mathbb E(C,o)
=
\mathbb Z\langle E(C,o)\rangle,
\qquad
\mathbb V(C)
=
\mathbb Z\langle V(C)\rangle.
\]
Consider the two-term chain complex
\[
0
\longrightarrow
\mathbb E(C,o)
\xrightarrow{\partial}
\mathbb V(C)
\longrightarrow
0,
\]
where
\[
\partial(e_l^v)=v,
\qquad
\partial(e_r^v)=v.
\]
Thus
\[
\partial(e_l^v-e_r^v)=0.
\]
{\color{currentbrown}
Generators corresponding to distinct double points belong to distinct
direct summands, and hence
\[
\ker\partial
=\bigoplus_{v\in V(C)}
\mathbb Z\langle e_l^v-e_r^v\rangle.
\]
Moreover, \(\partial\) is surjective.  Therefore, define
\[
\mathcal H_{\mathrm{edge}}(C,o)
:=
H_1\bigl(
\mathbb E(C,o)\xrightarrow{\partial}\mathbb V(C)
\bigr)
\cong
\bigoplus_{v\in V(C)}
\mathbb Z\langle e_l^v-e_r^v\rangle.
\]
On the other hand,
\[
H_0\bigl(
\mathbb E(C,o)\xrightarrow{\partial}\mathbb V(C)
\bigr)=0.
\]
}

For each double point \(v\), let
\[
z_v=[e_l^v-e_r^v]
\in
\mathcal H_{\mathrm{edge}}(C,o).
\]
The concentration of the edge complex in degrees \(1\) and \(0\) makes
\(e_l^v-e_r^v\) the unique cycle representative of \(z_v\).  Its two
terms carry the internal data
\[
\left(0,j(v)+\frac12\right)
\qquad\text{and}\qquad
\left(1,j(v)-\frac12\right),
\]
respectively.  The first entry is an internal sign label, independent
of the chain degree of the edge complex; these labels record the two
local terms carried by the canonical cycle.

The canonical basis above defines a signed \(q\)-character
\[
\operatorname{ch}^{\mathrm{edge}}_q
\colon
\mathcal H_{\mathrm{edge}}(C,o)
\longrightarrow
\mathbb Z[q^{1/2},q^{-1/2}],
\qquad
\operatorname{ch}^{\mathrm{edge}}_q(z_v)
=
q^{j(v)+1/2}-q^{j(v)-1/2}.
\]
This signed character is determined by the unique cycle representative of
each edge-homology class and evaluates its two internal labels with opposite
signs.

Set
\[
z_C:=\sum_{v\in V(C)}z_v.
\]
Then Equation~\eqref{eq:vertex-specialization} takes the
edge-to-vertex form
\[
\operatorname{ch}^{\mathrm{edge}}_q(z_C)
=
\operatorname{grank}_{t,q}\mathcal V(C,o)\big|_{t=-1}.
\]

For prescribed turning magnitudes
\(\beta=(\beta_v)_{v\in V(C)}\), define the
\emph{weighted signed \(q\)-character}
\[
\operatorname{ch}^{\mathrm{edge}}_{q,\beta}
\colon
\mathcal H_{\mathrm{edge}}(C,o)
\longrightarrow
\mathbb R[q^{1/2},q^{-1/2}]
\]
on the canonical basis by
\[
\operatorname{ch}^{\mathrm{edge}}_{q,\beta}(z_v)
:=
\frac{\beta_v}{2\pi}
\left(q^{j(v)+1/2}-q^{j(v)-1/2}\right).
\]

\begin{proposition}[Weighted signed \(q\)-character]
\label{prop:edge-curvature}
Let \((C,o)\) be a generic oriented one-component plane curve, and
let \(\mathcal H_{\mathrm{edge}}(C,o)\) be the edge homology
constructed from the two-term complex defined above in
Subsection~\ref{subsec:edge}.  \purplechange{It satisfies}
\[
{\color{revisionpurple}
H_1\cong\bigoplus_{v\in V(C)}
\mathbb Z\langle[e_l^v-e_r^v]\rangle,
\qquad H_0=0.}
\]
Suppose
that, at every \(v\in V(C)\), the orientation-preserving smoothing
produces a positive turning of magnitude \(\beta_v\) on the local arc
with averaged Alexander index \(j(v)+1/2\) and a negative turning of
the same magnitude on the local arc with averaged Alexander index
\(j(v)-1/2\).  Then the local change of the normalized quantized
curvature integral produced by smoothing \(v\) is
\[
\operatorname{ch}^{\mathrm{edge}}_{q,\beta}(z_v)
=
\frac{\beta_v}{2\pi}
\left(
q^{j(v)+1/2}-q^{j(v)-1/2}
\right).
\]
The total change produced by smoothing all double points is
\[
\operatorname{ch}^{\mathrm{edge}}_{q,\beta}(z_C)
=
\frac{1}{2\pi}
\sum_{v\in V(C)}
\beta_v
\left(
q^{j(v)+1/2}-q^{j(v)-1/2}
\right).
\]

\end{proposition}

\bluechange{Its local-weight calculation is given in
Section~\ref{proof:edge}.}

\purplechange{Proposition~\ref{prop:edge-curvature} proves
Part~{\rm (\ref{thm:main-edge})} of Theorem~\ref{thm:main}.
}

\subsection{The canonical two-state homology for an unoriented curve}
\label{subsec:two-state}

We now regard \(C\) as an unoriented one-component curve.  Choose
temporarily one of its two orientations \(o\).  Combining the
smoothing-circle homologies associated with the two orientations,
define
\[
\mathcal H_{\mathrm{circle}}^{\mathrm{un}}(C)
=
\mathcal H_{\mathrm{circle}}(C,o)
\oplus
\mathcal H_{\mathrm{circle}}(C,-o).
\]
Its triply graded components are defined by
\[
\mathcal H_{\mathrm{circle}}^{\mathrm{un},i,j,\varepsilon}(C)
:=
\mathcal H_{\mathrm{circle}}^{i,j,\varepsilon}(C,o)
\oplus
\mathcal H_{\mathrm{circle}}^{i,j,\varepsilon}(C,-o).
\]
Thus the superscript \((i,j,\varepsilon)\) on the left denotes the
same displayed degree in each orientation-dependent summand.
Orientation reversal gives
\[
\ind_{-o}(\widetilde C_{-o,\latestchange{\ell}})
=
-\ind_o(\widetilde C_{o,\latestchange{\ell}}),
\qquad
\rot_{-o}(\widetilde C_{-o,\latestchange{\ell}})
=
-\rot_o(\widetilde C_{o,\latestchange{\ell}}),
\]
and therefore
\[
\varepsilon_{-o}(\widetilde C_{-o,\latestchange{\ell}})
=
1-\varepsilon_o(\widetilde C_{o,\latestchange{\ell}}).
\]
Consequently, orientation reversal canonically identifies
\[
\mathcal H_{\mathrm{circle}}^{i,j,\varepsilon}(C,o)
\cong
\mathcal H_{\mathrm{circle}}^{i,-j,1-\varepsilon}(C,-o).
\]
It induces an involution \(\tau\) on the total two-state homology with
\[
\tau\bigl(
\mathcal H_{\mathrm{circle}}^{\mathrm{un},i,j,\varepsilon}(C)
\bigr)
=
\mathcal H_{\mathrm{circle}}^{\mathrm{un},i,-j,1-\varepsilon}(C).
\]
Changing the temporary reference orientation merely reverses the
displayed order of the two factors in each fixed
\((i,j,\varepsilon)\)-component.  Hence the triply graded group is
independent of that choice.

Define the triply graded-rank polynomial of the two-state homology by
\begin{equation}
\label{eq:two-state-circle-sum}
\begin{aligned}
\operatorname{grank}_{t,q,u}
\mathcal H_{\mathrm{circle}}^{\mathrm{un}}(C)
:={}&
\operatorname{grank}_{t,q,u}\mathcal H_{\mathrm{circle}}(C,o)
+
\operatorname{grank}_{t,q,u}\mathcal H_{\mathrm{circle}}(C,-o).
\end{aligned}
\end{equation}
Using either orientation \(o\) as a reference, this becomes
\[
\operatorname{grank}_{t,q,u}\mathcal H_{\mathrm{circle}}^{\mathrm{un}}(C)
=
\sum_{i,j,\varepsilon}
\operatorname{rank}\mathcal H_{\mathrm{circle}}^{i,j,\varepsilon}
(C,o)t^i
\left(q^ju^\varepsilon+q^{-j}u^{1-\varepsilon}\right).
\]

{\color{lastbrown}
Put
\[
\begin{aligned}
G_o(q)&:=
\left.\operatorname{grank}_{t,q,u}
\mathcal H_{\mathrm{circle}}(C,o)\right|_{t=u=-1},\\
G_{\mathrm{un}}(q)&:=
\left.\operatorname{grank}_{t,q,u}
\mathcal H_{\mathrm{circle}}^{\mathrm{un}}(C)\right|_{t=u=-1}.
\end{aligned}
\]
Orientation reversal gives
\(G_{\mathrm{un}}(q)=G_o(q)-G_o(q^{-1})\), and hence, for every
\(m\geq0\),
\[
\left.D_q^mG_{\mathrm{un}}(q)\right|_{q=1}
=\bigl(1-(-1)^m\bigr)
\left.D_q^mG_o(q)\right|_{q=1}.
\]
}

\begin{proposition}[The canonical two-state homology]
\label{prop:two-state}
\purplechange{For a generic unoriented one-component plane curve \(C\), the
two-state homology}
\begin{equation*}
{\color{revisionpurple}
\mathcal H_{\mathrm{circle}}^{\mathrm{un}}(C)
=\mathcal H_{\mathrm{circle}}(C,o)
\oplus\mathcal H_{\mathrm{circle}}(C,-o)
\tag{11}}
\end{equation*}
\purplechange{is independent of the temporary choice of orientation \(o\) as
a triply graded group and is invariant under direct self-tangency and weak
triple-point perestroikas.  Its canonical orientation-exchange involution
\(\tau\) sends degree \((i,j,\varepsilon)\) to
\((i,-j,1-\varepsilon)\).  Put}
\[
{\color{revisionpurple}
\begin{aligned}
G_o(q)&:=\left.\operatorname{grank}_{t,q,u}
\mathcal H_{\mathrm{circle}}(C,o)\right|_{t=u=-1},\\
G_{\mathrm{un}}(q)&:=\left.\operatorname{grank}_{t,q,u}
\mathcal H_{\mathrm{circle}}^{\mathrm{un}}(C)\right|_{t=u=-1}.
\end{aligned}}
\]
\purplechange{Then its decategorification is the antisymmetrization}
\[
{\color{revisionpurple}
G_{\mathrm{un}}(q)=G_o(q)-G_o(q^{-1}).}
\]
\purplechange{For every \(m\geq0\),}
\begin{equation*}
{\color{revisionpurple}
\left.D_q^mG_{\mathrm{un}}(q)\right|_{q=1}
=\bigl(1-(-1)^m\bigr)
\left.D_q^mG_o(q)\right|_{q=1}.
\tag{19}}
\end{equation*}
\purplechange{Thus the odd moments double and the even moments cancel.  In
particular,}
\begin{equation*}
{\color{revisionpurple}
J^-(C)=1-G_{\mathrm{un}}'(1),
\tag{20}}
\end{equation*}
\purplechange{and}
\begin{equation*}
{\color{revisionpurple}
J^+(C)=1-G_{\mathrm{un}}'(1)+\#V(C).
\tag{21}}
\end{equation*}
\end{proposition}

\bluechange{The generatorwise involution, the odd/even Euler-derivative
identity, and equivariance under the allowed perestroikas are proved in
Section~\ref{proof:two-state}.}
\purplechange{Proposition~\ref{prop:two-state} proves
Part~{\rm (\ref{thm:main-unoriented})} of Theorem~\ref{thm:main}.}

\subsection{Infinitely many pairs distinguished by the homological
Seifert-state lifts}
\label{subsec:infinite-pairs}

\begin{proposition}[Infinitely many pairs with identical polynomial data]
\label{prop:infinite-pairs}
There exists an infinite family
\[
\bigl\{((C_n,o_n),(C'_n,o'_n))\bigr\}_{n\ge1}
\]
satisfying Part~{\rm (\ref{thm:main-infinite-pairs})} of
Theorem~\ref{thm:main}.
\end{proposition}

\begin{figure}[ht]
\centering
\includegraphics[width=0.98\textwidth]{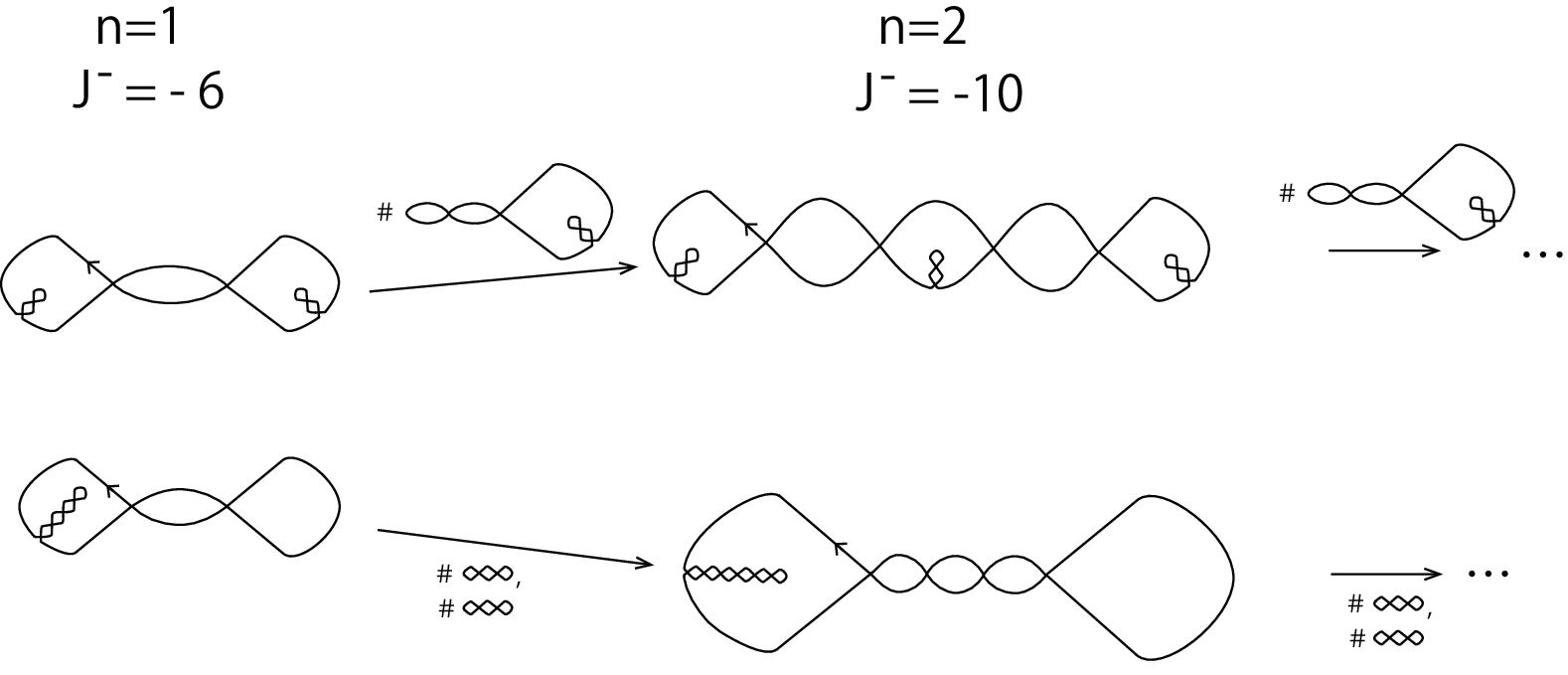}
\caption{\newchange{The diagram of the infinite family.  For each
\(n\ge1\), the two one-component curves have the same number of double
points.}}
\label{fig:infinite-pairs}
\end{figure}

\bluechange{The complete circle table, component-homology table, branch
calculation, and Poincar\'e-polynomial calculation are given in
Section~\ref{proof:strictness}.}

\section{Componentwise refinements and Poincar\'e evaluations}
\label{sec:component-refinements}

The actual region summands and the branches of the rooted dual
tree are the primary data.  This section records two numerical evaluations
without replacing those decompositions by the resulting integers.

\subsection{The component grading of the region homology}

For an actual region \(R\), set
\[
\latestchange{k(R)}:=\#\pi_0(\partial R)-1\in\mathbb Z_{\geq0}.
\]
For a bounded planar region this integer equals
\(\operatorname{rank}H_1(R;\mathbb Z)\).  Grouping the actual summands by
this integer gives
\[
\mathcal H_{\mathrm{region}}^{i,j,k}(C,o)
:=
\bigoplus_{\substack{R\in\pi_0(R_j(C,o))\\ \latestchange{k(R)}=k}}
H_i(R;\mathbb Z).
\]
Its Poincar\'e polynomial is
\begin{equation}
\label{eq:component-Poincare}
\mathcal P_{\mathrm{region}}(C,o;t,q,v)
:=
\sum_j\sum_{R\in\pi_0(R_j(C,o))}\sum_i
\operatorname{rank}H_i(R;\mathbb Z)t^iq^jv^{\latestchange{k(R)}}.
\end{equation}
It satisfies
\[
\mathcal P_{\mathrm{region}}(C,o;t,q,1)
=\operatorname{grank}_{t,q}\mathcal H_{\mathrm{region}}(C,o),
\qquad
\mathcal P_{\mathrm{region}}(C,o;-1,q,1)=P_C(q).
\]
This numerical grading can forget the actual face labels; it is an
evaluation of the componentwise direct-sum decomposition, not its
replacement.

\begin{proposition}[Component Poincar\'e evaluation]
\label{prop:component-Poincare}
\purplechange{The polynomial}
\begin{equation*}
{\color{revisionpurple}\adjustbox{max width=0.90\linewidth}{$\displaystyle
\mathcal P_{\mathrm{region}}(C,o;t,q,v)
:=\sum_j\sum_{R\in\pi_0(R_j(C,o))}\sum_i
\operatorname{rank}H_i(R;\mathbb Z)t^iq^jv^{k(R)}.$}}
\tag{12}
\end{equation*}
\purplechange{is preserved by direct self-tangency and weak
triple-point perestroikas.  For every \(n\geq1\), the two curves
\((C_n,o_n)\) and \((C'_n,o'_n)\) in
Figure~\ref{fig:infinite-pairs} have the following coefficients of
\(q^1\), respectively:}
{\color{latestbrown}
\[
(n+1)(1+2t)v^2
\qquad\text{and}\qquad
\bigl(1+(2n+2)t\bigr)v^{2n+2}+n.
\]
Since these coefficients are unequal,
\[
\mathcal P_{\mathrm{region}}(C_n,o_n;t,q,v)
\ne
\mathcal P_{\mathrm{region}}(C'_n,o'_n;t,q,v).
\]
}
Hence this evaluation distinguishes every displayed pair,
although the two coefficients become equal after \(v=1\).
\end{proposition}

\bluechange{The complete component calculation is given in
Section~\ref{proof:strictness}.}

\subsection{The branch decomposition of the smoothing-circle homology}

\rthreechange{For each smoothing circle \(K\in\pi_0(\widetilde C_o)\), retain the unordered pair of actual complementary regions \(R\in\pi_0(\mathbb R^2\setminus\widetilde C_o)\) satisfying \(K\subset\partial R\).}

{\color{branchblue}
Let \(T(\widetilde C_o)\) be the dual tree: its vertices are the actual
components of \(\mathbb R^2\setminus\widetilde C_o\), and every smoothing
circle gives the edge joining its two boundary regions.  Root the tree at
the vertex \(v_\infty\) of the unbounded region.  Its branches are
\[
B\in\pi_0\!\left(T(\widetilde C_o)\setminus\{v_\infty\}\right).
\]
Assign a smoothing circle to \(B\) when the non-root endpoint of its dual
edge lies in \(B\); this includes the edge joining \(B\) to the root.
This gives a partition of all smoothing circles.  \finalchange{Denote the
resulting branch-assignment map by
\[
\beta_o\colon\pi_0(\widetilde C_o)\longrightarrow
\pi_0\!\left(T(\widetilde C_o)\setminus\{v_\infty\}\right),
\qquad K\longmapsto\beta_o(K).
\]}
Define
\[
\mathcal H_{\mathrm{circle}}^{i,j,\varepsilon}(C,o;B)
=
\begin{cases}
\displaystyle
\bigoplus_{\substack{K\in\pi_0(\widetilde C_o),\ \finalchange{\beta_o(K)=B}\\
\ind_o(K)=j,\ \varepsilon(K)=\varepsilon}}
\mathbb Z\langle[K]\rangle,&i=1,\\[3mm]
0,&i\ne1.
\end{cases}
\]
The invariant is the branch-indexed family
\begin{equation}
\label{eq:branch-circle-decomposition}
\left(\mathcal H_{\mathrm{circle}}^{i,j,\varepsilon}(C,o;B)\right)_
{B\in\pi_0(T(\widetilde C_o)\setminus\{v_\infty\})}.
\end{equation}
It is compared up to permutation of the actual branches.  Forgetting the
branch index and taking the direct sum recovers the ordinary
smoothing-circle homology.  Counting the circles assigned to each branch
gives the numerical evaluation
\begin{equation}
\label{eq:branch-circle-Poincare}
{\color{thisbrown}\resizebox{0.98\textwidth}{!}{$\displaystyle
\mathcal P_{\mathrm{branch}}(C,o;t,q,u,w):=
\sum_{B\in\pi_0(T(\widetilde C_o)\setminus\{v_\infty\})}
\left(\sum_{\substack{K\in\pi_0(\widetilde C_o)\\\beta_o(K)=B}}
tq^{\ind_o(K)}u^{\varepsilon(K)}\right)
w^{\#\beta_o^{-1}(B)}.$}}
\end{equation}
At \(w=1\), this is
\(\operatorname{grank}_{t,q,u}\mathcal H_{\mathrm{circle}}(C,o)\).
The full branch decomposition \eqref{eq:branch-circle-decomposition} is the
primary structure; \eqref{eq:branch-circle-Poincare} is one numerical
evaluation of it.

}

\begin{proposition}[Invariance of the branch decomposition]
\label{prop:branch-invariance}
\branchtext{The ambient isotopy of
Lemma~\ref{lem:smoothed-configuration-invariance} fixes the unbounded
region and induces an isomorphism of rooted dual trees.  It preserves the
branch-indexed family}
\begin{equation*}
{\color{revisionpurple}
\left(\mathcal H_{\mathrm{circle}}^{i,j,\varepsilon}(C,o;B)\right)_
{B\in\pi_0(T(\widetilde C_o)\setminus\{v_\infty\})}}
\tag{13}
\end{equation*}
\purplechange{and the branch Poincar\'e polynomial}
\begin{equation*}
{\color{revisionpurple}\adjustbox{max width=0.90\linewidth}{$\displaystyle
\mathcal P_{\mathrm{branch}}(C,o;t,q,u,w):=
\sum_{B\in\pi_0(T(\widetilde C_o)\setminus\{v_\infty\})}
\left(\sum_{\substack{K\in\pi_0(\widetilde C_o)\\\beta_o(K)=B}}
tq^{\ind_o(K)}u^{\varepsilon(K)}\right)
w^{\#\beta_o^{-1}(B)}.$}}
\tag{14}
\end{equation*}
\end{proposition}

\bluechange{The proof is given in Section~\ref{proof:branch-invariance}.}
\purplechange{Propositions~\ref{prop:infinite-pairs},
\ref{prop:component-Poincare}, and \ref{prop:branch-invariance},
together with the component and branch Poincar\'e calculations verified in
Section~\ref{proof:strictness}, prove
Part~{\rm (\ref{thm:main-infinite-pairs})} of
Theorem~\ref{thm:main}.}

\section{The homological face-state-sum model and averaging}
\label{sec:discussion}

{\color{faceorange}
Viro placed his Euler-integral polynomial in the quantum-invariant context:
after the substitution \(q=e^h\), its Taylor coefficients are finite-degree
invariants, and its changes under perestroikas resemble skein relations.
He described the Euler-integral formula as an analogue of face state-sum
formulas for quantum knot polynomials
\cite[Section~5.3]{Viro1996}.

Here this analogy is realized by an explicit sum of products of local face
weights.  Fix an orientation state \(s\in\{o,-o\}\).  On the dual tree
\(T(\widetilde C_s)\), use \(\ind_s(R)\) as the Alexander face coloring,
normalized by \(\ind_s(R_\infty)=0\).
Across an edge, the two colors differ by one with the sign prescribed by
the oriented smoothing circle.  Since the dual graph is a tree, this
normalization determines the coloring.

A state is the colored dual tree together with one marked actual face
\(R_*\in\pi_0(\mathbb R^2\setminus\widetilde C_s)\).  For this state,
assign to every face \(R\) the local weight
\[
W_{R_*}(R)
=
\begin{cases}
\displaystyle
\left(\sum_i\operatorname{rank}H_i(R;\mathbb Z)t^i\right)
q^{\ind_s(R)}[R],&R=R_*,\\[3mm]
1,&R\ne R_*.
\end{cases}
\]
The resulting partition function has the traditional state-sum form
\emph{sum over states of products of local weights} and is
\begin{align}
\mathcal Z_{\mathrm{face}}(C,s;t,q)
&:=
\sum_{R_*\in\pi_0(\mathbb R^2\setminus\widetilde C_s)}
\prod_R W_{R_*}(R)\notag\\
&=
\sum_j\sum_{R\in\pi_0(R_j(C,s))}\sum_i
\operatorname{rank}H_i(R;\mathbb Z)t^iq^j[R].
\label{eq:face-state-sum}
\end{align}
Thus the face variable is an actual region, its coloring is the Alexander
index, and its Euler-characteristic weight is lifted to its singular
homology.  The natural homomorphism \([R]\mapsto1\) sends
\eqref{eq:face-state-sum} to
\(\operatorname{grank}_{t,q}\mathcal H_{\mathrm{region}}(C,s)\);
\thischange{the further specialization \(t=-1\) gives the corresponding
orientation-state specialization stated in Theorem~\ref{thm:main}.}

For an unoriented curve, the two orientation states give
\begin{equation}
\label{eq:two-state-face-sum}
\mathcal Z_{\mathrm{face}}^{\mathrm{un}}(C;t,q)
:=
\sum_{s\in\{o,-o\}}
\sum_j\sum_{R\in\pi_0(R_j(C,s))}\sum_i
\operatorname{rank}H_i(R;\mathbb Z)t^iq^j[s,R].
\end{equation}
{\color{currentbrown}
For every \(i,j\), define the unoriented region homology by
\begin{align}
\mathcal H_{\mathrm{region}}^{\mathrm{un},i,j}(C)
&:=\bigoplus_{s\in\{o,-o\}}
\mathcal H_{\mathrm{region}}^{i,j}(C,s)\notag\\
&=\mathcal H_{\mathrm{region}}^{i,j}(C,o)
\oplus\mathcal H_{\mathrm{region}}^{i,j}(C,-o).
\label{eq:unoriented-region-homology}
\end{align}
The underlying actual regions are independent of their orientation state,
whereas
\[
\ind_{-o}(R)=-\ind_o(R).
\]
Consequently,
\[
R_j(C,-o)=R_{-j}(C,o),
\qquad
\mathcal H_{\mathrm{region}}^{i,j}(C,-o)
\cong
\mathcal H_{\mathrm{region}}^{i,-j}(C,o)
\]
component by component.  Changing the temporary orientation exchanges the
two direct-sum factors in \eqref{eq:unoriented-region-homology}; hence this
two-state homology is independent of that choice.

\begin{propositionstar}[Orientation independence of the two-state region homology]
\rthreechange{The group \(\mathcal H_{\mathrm{region}}^{\mathrm{un},i,j}(C)\) defined in \eqref{eq:unoriented-region-homology} is independent of the temporary orientation of the unoriented curve \(C\); reversing that orientation exchanges its two direct-sum factors and sends Alexander degree \(j\) to \(-j\).}
\end{propositionstar}

The homomorphism
\([s,R]\mapsto1\) sends \eqref{eq:two-state-face-sum} to
\(\operatorname{grank}_{t,q}
\mathcal H_{\mathrm{region}}^{\mathrm{un}}(C)\).

{\color{lastbrown}
The corresponding component Poincar\'e evaluation is
\begin{equation}
\mathcal P_{\mathrm{region}}^{\mathrm{un}}(C;t,q,v)
:=\sum_{s\in\{o,-o\}}
\mathcal P_{\mathrm{region}}(C,s;t,q,v).
\label{eq:unoriented-component-Poincare}
\end{equation}
It is likewise independent of the temporary orientation.
}

{\color{lastbrown}
The relation with the smoothing circles is recorded one actual connected
region at a time, without choosing a cellular model.  Every smoothing
circle \(K\subset\partial R\) has a collar push-off
\(\gamma_K\subset R\), oriented as a boundary cycle.  If \(R\) is bounded
and has boundary circles \(K_0,\ldots,K_k\), these cycles give the standard
presentation
\[
H_1(R;\mathbb Z)\cong
\frac{\bigoplus_{\nu=0}^{k}\mathbb Z\langle[\gamma_\nu]\rangle}
{\mathbb Z\langle[\gamma_0]+\cdots+[\gamma_k]\rangle}
\cong\mathbb Z^k.
\]
A standard deformation retraction onto a bouquet of the \(k\) inner loops
shows that the outer loop is homotopically dependent; its homology class is
eliminated by the displayed boundary relation.
Moreover,
\[
H_0(R;\mathbb Z)\cong\mathbb Z\langle[p_R]\rangle,
\qquad H_i(R;\mathbb Z)=0\quad(i\geq2).
\]
Thus, after summing over \(R\in\pi_0(R_j(C,o))\), the zero-dimensional
homology has one generator \([p_R]\) for each connected region.  If \(R\)
is unbounded with \(b\) actual boundary circles, their collar push-offs
form a basis of \(H_1(R;\mathbb Z)\cong\mathbb Z^b\), with the same
descriptions of \(H_0\) and the higher homology.

Orientation reversal leaves every actual region, point cycle \([p_R]\),
and collar cycle \([\gamma_K]\) unchanged, while it sends its Alexander
degree from \(j\) to \(-j\).  Hence both the zero-cycles indexed by
connected regions and every boundary-cycle generator are matched
component by component between the two summands of
\(\mathcal H_{\mathrm{region}}^{\mathrm{un}}(C)\).  Together with this
cycle correspondence, Equations~\eqref{eq:face-state-sum} and
\eqref{eq:two-state-face-sum} give the promised concrete homological
face-state-sum model rather than only an equality of its final
polynomials.
}
}

Averaging connects this face model with the curvature model.  The Alexander
index is constant on every complementary face.  Its average across a
smoothing circle gives the half-integral degree of the circle generator,
and its average at a double point gives the degrees of the vertex and edge
terms.  Hopf's Umlaufsatz turns the curvature integral of each smoothed
circle into its rotation sign.  Locally, the edge class
\([e_l^v-e_r^v]\) carries the two averaged degrees with opposite bending
signs, while the vertex complex places the same degrees in opposite
homological parities.

Viro's use of Rokhlin's complex-orientation formula gives a separate
geometric parallel.  For a real algebraic curve of Klein type~I, complex
conjugation exchanges the two complex orientations.  The involution on the
two-state homology likewise exchanges the two orientations and sends
\((j,\varepsilon)\) to \((-j,1-\varepsilon)\).
}

\begin{proposition}[Homological face-state-sum specialization]
\label{prop:face-state-sum}
\purplechange{For \(s\in\{o,-o\}\), the marked-face state sum is}
\begin{equation*}
\begin{aligned}
\mathcal Z_{\mathrm{face}}(C,s;t,q)
&=\sum_{R_*\in\pi_0(\mathbb R^2\setminus\widetilde C_s)}
\prod_R W_{R_*}(R)\\
&=\sum_j\sum_{R\in\pi_0(R_j(C,s))}\sum_i
\operatorname{rank}H_i(R;\mathbb Z)t^iq^j[R].
\end{aligned}
\tag{15}
\end{equation*}
\purplechange{For an unoriented curve, its two-state version is}
\begin{equation*}
{\color{revisionpurple}
\mathcal Z_{\mathrm{face}}^{\mathrm{un}}(C;t,q)
=\sum_{s\in\{o,-o\}}\sum_j
\sum_{R\in\pi_0(R_j(C,s))}\sum_i
\operatorname{rank}H_i(R;\mathbb Z)t^iq^j[s,R].}
\tag{16}
\end{equation*}
\purplechange{Equations (15) and (16) are invariant under direct
self-tangency and weak triple-point perestroikas.  The homomorphism
\([R]\mapsto1\) sends (15) to
\(\operatorname{grank}_{t,q}\mathcal H_{\mathrm{region}}(C,s)\), and
the further specialization \(t=-1\) gives \(P_C(q)\) for \(s=o\) and
\(P_C(q^{-1})\) for \(s=-o\).  Equation (16) is independent of the
temporary orientation.}
\end{proposition}

\bluechange{The specialization and invariance checks are given in
Section~\ref{proof:face-state-sum}.}
\purplechange{Lemma~\ref{lem:smoothed-configuration-invariance} and
Propositions~\ref{prop:component-refinement-invariance},
\ref{prop:branch-invariance}, and \ref{prop:face-state-sum} prove
Part~{\rm (\ref{thm:main-components})} of Theorem~\ref{thm:main}.}

\section{Proofs of the statements}
\label{sec:proofs}

\bluechange{This section collects the complete proofs so that the chain of
constructions in Sections~\ref{sec:homological-refinements}--
\ref{sec:discussion} can be read without interruption.  The classical
inputs are labeled \((\mathrm{I1})\)--\((\mathrm{I5})\); equations
constructed in this paper retain their ordinary numbering.}

\thischange{Whenever a numbered manuscript formula is obtained below, its
authoritative number is repeated at the right margin.  The symbol
\(\because\) records the input identity used at that step.}

\subsection{Region homology and its decategorification}
\label{proof:region-decategorification}

\begin{proof}[Proof of Proposition~\ref{prop:region-decategorification}]
The finite disjoint union
\[
R_j(C,o)=\coprod_{R\in\pi_0(R_j(C,o))}R
\]
gives a direct sum already at the chain level:
\[
C_i(R_j(C,o);\mathbb Z)
=
\bigoplus_{R\in\pi_0(R_j(C,o))}C_i(R;\mathbb Z).
\]
The boundary operator preserves every summand.  Taking homology gives
\begin{equation*}
H_i(R_j(C,o);\mathbb Z)
\cong
\bigoplus_{R\in\pi_0(R_j(C,o))}H_i(R;\mathbb Z).
\tag{\thischange{1}}
\end{equation*}
The factors are indexed by the actual regions, so this is the canonical
decomposition in \eqref{eq:main-component-decomposition}.

By definition,
\[
\operatorname{grank}_{t,q}\mathcal H_{\mathrm{region}}(C,o)
=
\sum_{i,j}
\operatorname{rank}\mathcal H_{\mathrm{region}}^{i,j}(C,o)t^iq^j.
\]
Therefore
\begin{align*}
\operatorname{grank}_{t,q}\mathcal H_{\mathrm{region}}(C,o)
\big|_{t=-1}
&=
\sum_j
\left(
\sum_i(-1)^i
\operatorname{rank}\mathcal H_{\mathrm{region}}^{i,j}(C,o)
\right)q^j\\
&=
\sum_j\chi(R_j(C,o))q^j\\
&=P_C(q)\qquad \thischange{\because\ \mathrm{(I1)}}
\tag{\thischange{2}}
\end{align*}
Applying \(D_q^n\) term by term and using \(D_q^n(q^j)=j^nq^j\) gives
\begin{align*}
&\left.
D_q^n
\left(
\operatorname{grank}_{t,q}\mathcal H_{\mathrm{region}}(C,o)
\big|_{t=-1}
\right)\right|_{q=1}\\
&\quad=
\left.
\sum_{i,j}(-1)^i
\operatorname{rank}\mathcal H_{\mathrm{region}}^{i,j}(C,o)
j^nq^j
\right|_{q=1}\\
&\quad=
\sum_{i,j}(-1)^ij^n
\operatorname{rank}\mathcal H_{\mathrm{region}}^{i,j}(C,o).
\tag{\thischange{8}}
\end{align*}
For \(n=2\), substitution into \((\mathrm{I2})\) yields
\[
J^-(C)
=
1-\sum_{i,j}(-1)^ij^2
\operatorname{rank}\mathcal H_{\mathrm{region}}^{i,j}(C,o).
\]
\end{proof}

\subsection{Smoothing-circle homology and its decategorification}
\label{proof:circle-decategorification}

\begin{proof}[Proof of Proposition~\ref{prop:circle-decategorification}]
The generator \([K]\in H_1(K;\mathbb Z)\) has degree
\[
\bigl(1,\ind_o(K),\varepsilon(K)\bigr)
\]
and hence contributes
\[
tq^{\ind_o(K)}u^{\varepsilon(K)}
\]
to the graded rank.  Summing over all circles gives the unspecialized
formula in the proposition.  At \(t=u=-1\), its coefficient becomes
\[
(-1)^1(-1)^{\varepsilon(K)}=(-1)^{1+\varepsilon(K)}.
\]
The sign check is
\[
\begin{array}{c|c|c}
\rot_o(K)&\varepsilon(K)&(-1)^{1+\varepsilon(K)}\\ \hline
+1&1&+1\\
-1&0&-1
\end{array}
\]
and therefore
\[
(-1)^{1+\varepsilon(K)}=\rot_o(K).
\]
\thischange{It follows that}
\begin{align*}
\operatorname{grank}_{t,q,u}\mathcal H_{\mathrm{circle}}(C,o)
\big|_{t=u=-1}
&=
\sum_K\rot_o(K)q^{\ind_o(K)}\\
&=I_q(\widetilde C_o)\qquad
\thischange{\because\ \mathrm{(I4)}}
\tag{\thischange{3}}
\end{align*}
Finally, applying \(D_q^n\) term by term gives
\begin{equation*}
\left.
D_q^n
\left(
\operatorname{grank}_{t,q,u}\mathcal H_{\mathrm{circle}}(C,o)
\big|_{t=u=-1}
\right)
\right|_{q=1}
=
\sum_K\rot_o(K)\bigl(\ind_o(K)\bigr)^n.
\tag{\thischange{9}}
\end{equation*}
\end{proof}

\subsection{Ambient isotopy and invariance}
\label{proof:smoothed-configuration-invariance}

\begin{proof}[Proof of Lemma~\ref{lem:smoothed-configuration-invariance}]
Choose a closed disk \(D\) containing the perestroika and no other
singularity.  Outside \(D\), the two curves agree.  Inside \(D\), perform
the orientation-preserving smoothing at every local double point.  For a
direct self-tangency or a weak triple-point move, the pairing of the
boundary endpoints before and after the move is the same.  Move the
smoothed arcs directly through \(D\), keeping a collar of \(\partial D\)
fixed.  The unchanged endpoint pairing allows the local motions to remain
disjoint.  They give an ambient isotopy of \(D\), fixed near its boundary,
and extension by the identity gives the required plane isotopy \(F_t\).

If \(x\) lies in a complementary region \(R\), then
\(t\mapsto F_t(x)\) never meets \(F_t(\widetilde C_o)\).  The winding
number of \(F_t(\widetilde C_o)\) about \(F_t(x)\) is a continuous
integer-valued function of \(t\), hence is constant.  Region indices and
their averages across circles are therefore preserved.  Each oriented
circle is carried through an isotopy of oriented embedded circles, so its
rotation sign is preserved.  Finally,
\[
F_1(\partial R)=\partial F_1(R),
\]
and hence
\[
K\subset\partial R
\quad\Longleftrightarrow\quad
F_1(K)\subset\partial F_1(R).
\]
\end{proof}

\label{proof:homological-invariance}
\begin{proof}[Proof of Proposition~\ref{prop:component-refinement-invariance}]
The lemma gives a bijection
\[
\pi_0(R_j(C,o))\longrightarrow\pi_0(R_j(C',o')),
\qquad
R\longmapsto F_1(R).
\]
For each actual region, \(F_1|_R\colon R\to F_1(R)\) is a homeomorphism,
and hence
\[
(F_1|_R)_*\colon H_i(R;\mathbb Z)
\longrightarrow H_i(F_1(R);\mathbb Z)
\]
is an isomorphism.  Taking the direct sum gives the displayed
componentwise region isomorphism.

On the circle side, define the map generator by generator:
\[
[K]\longmapsto[F_1(K)].
\]
The lemma shows that this map preserves homological degree, averaged index,
rotation-sign degree, and every relation \(K\subset\partial R\).  It is
bijective, and therefore is the asserted triply graded isomorphism with
the retained region data.
\end{proof}

\subsection{Comparison and the smoothing-circle formula for \(J^-\)}
\label{proof:comparison}

\begin{proof}[Proof of Proposition~\ref{prop:comparison}]
Put \(a(q)=q^{1/2}-q^{-1/2}\).  The classical comparison input
\((\mathrm{I5})\) reads
\begin{equation*}
P_C(q)
=
1+a(q)I_q(C)
+\frac12\sum_{v\in V(C)}q^{\ind_o(v)}a(q)^2.
\tag*{\thischange{\(\because\ \mathrm{(I5)}\)}}
\end{equation*}
Substitute the smoothing input \((\mathrm{I3})\):
\begin{align*}
P_C(q)
&=
1+a(q)\left(
I_q(\widetilde C_o)
-\frac12\sum_vq^{\ind_o(v)}a(q)
\right)
+\frac12\sum_vq^{\ind_o(v)}a(q)^2\\
&=
1+a(q)I_q(\widetilde C_o)
-\frac12\sum_vq^{\ind_o(v)}a(q)^2
+\frac12\sum_vq^{\ind_o(v)}a(q)^2\\
&=1+a(q)I_q(\widetilde C_o)\qquad
\thischange{\because\ \mathrm{(I3)},\mathrm{(I5)}}
\tag{\thischange{6}}
\end{align*}
\thischange{Substituting \eqref{eq:region-specialization} and
\eqref{eq:circle-specialization-detail} gives}
{\color{thisbrown}
\begin{equation*}
\begin{aligned}
\operatorname{grank}_{t,q}\mathcal H_{\mathrm{region}}(C,o)
\big|_{t=-1}
={}&1+\left(q^{1/2}-q^{-1/2}\right)\\[-1mm]
&\quad\cdot
\operatorname{grank}_{t,q,u}\mathcal H_{\mathrm{circle}}(C,o)
\big|_{t=u=-1}.
\end{aligned}
\qquad\because\ \mathrm{(2)},\mathrm{(3)},\mathrm{(6)}
\tag{5}
\end{equation*}
}
\end{proof}

\label{proof:circle-Jminus}
\begin{proof}[Proof of Corollary~\ref{cor:circle-Jminus}]
Write
\[
a(q)=q^{1/2}-q^{-1/2},
\qquad
b(q)=I_q(\widetilde C_o).
\]
Equation~\eqref{eq:main-comparison-classical} gives
\(P_C(q)=1+a(q)b(q)\).  The required Euler derivatives are
\begin{align*}
D_qa(q)
&=
q\frac{d}{dq}\left(q^{1/2}-q^{-1/2}\right)\\
&=
\frac12\left(q^{1/2}+q^{-1/2}\right),
\end{align*}
and
\begin{align*}
D_q^2a(q)
&=
q\frac{d}{dq}
\left[
\frac12\left(q^{1/2}+q^{-1/2}\right)
\right]\\
&=
\frac14\left(q^{1/2}-q^{-1/2}\right).
\end{align*}
Consequently,
\[
a(1)=0,\qquad (D_qa)(1)=1,\qquad (D_q^2a)(1)=0.
\]
Using the product rule twice,
\begin{align*}
D_q^2(ab)
&=
D_q\bigl((D_qa)b+a(D_qb)\bigr)\\
&=
(D_q^2a)b+2(D_qa)(D_qb)+a(D_q^2b).
\end{align*}
Evaluation at \(q=1\) gives
\begin{align*}
\left.D_q^2P_C(q)\right|_{q=1}
&=
0\cdot b(1)+2\cdot1\cdot(D_qb)(1)
+0\cdot(D_q^2b)(1)\\
&=
2(D_qb)(1)
=
2b'(1).
\end{align*}
\thischange{Substitution into \((\mathrm{I2})\), followed by
\eqref{eq:circle-specialization-detail}, gives}
\begin{align*}
J^-(C)
&=1-2I_q'(\widetilde C_o)\big|_{q=1}
\qquad\thischange{\because\ \mathrm{(I2)},\mathrm{(6)}}\\
&=1-2\left.\frac{d}{dq}\left(
\operatorname{grank}_{t,q,u}\mathcal H_{\mathrm{circle}}(C,o)
\big|_{t=u=-1}\right)\right|_{q=1}
\qquad\thischange{\because\ \mathrm{(3)}}
\tag{\thischange{4}}
\end{align*}
\end{proof}

\subsection{The vertex and edge constructions}
\label{proof:full-LP}

\begin{proof}[Proof of Proposition~\ref{prop:full-LP}]
The two generators at \(v\) contribute
\[
q^{j(v)+1/2}+tq^{j(v)-1/2}
\]
to the graded rank.  Therefore
\[
\operatorname{grank}_{t,q}\mathcal V(C,o)
=
\sum_{v\in V(C)}
\left(q^{j(v)+1/2}+tq^{j(v)-1/2}\right).
\]
At \(t=-1\),
\begin{align*}
\operatorname{grank}_{t,q}\mathcal V(C,o)\big|_{t=-1}
&=
\sum_v\left(q^{j(v)+1/2}-q^{j(v)-1/2}\right)\\
&=
\sum_vq^{j(v)}
\left(q^{1/2}-q^{-1/2}\right).
\tag{\thischange{10}}
\end{align*}
\thischange{Together with \eqref{eq:circle-specialization-detail}, the classical
smoothing identity \((\mathrm{I3})\) becomes}
{\color{thisbrown}
\begin{equation*}
\begin{aligned}
I_q(C)={}&
\operatorname{grank}_{t,q,u}\mathcal H_{\mathrm{circle}}(C,o)
\big|_{t=u=-1}\\
&-\frac12\operatorname{grank}_{t,q}\mathcal V(C,o)\big|_{t=-1}
\qquad\because\ \mathrm{(3)},\mathrm{(10)},\mathrm{(I3)}.
\end{aligned}
\tag{7}
\end{equation*}
}
\purplechange{This proves the first assertion.}

For the derivative, each vertex contributes
\begin{align*}
\left.
\frac{d}{dq}
\left(q^{j(v)+1/2}-q^{j(v)-1/2}\right)
\right|_{q=1}
&=
\left(j(v)+\frac12\right)
-\left(j(v)-\frac12\right)\\
&=1.
\end{align*}
Hence
\[
\left.
\frac{d}{dq}
\left(
\operatorname{grank}_{t,q}\mathcal V(C,o)\big|_{t=-1}
\right)
\right|_{q=1}
=
\#V(C).
\]
Differentiating \((\mathrm{I3})\) gives
\[
I_q'(C)\big|_{q=1}
=
I_q'(\widetilde C_o)\big|_{q=1}
-\frac12\#V(C).
\]
Therefore
\begin{align*}
1-2I_q'(C)\big|_{q=1}
&=
1-2I_q'(\widetilde C_o)\big|_{q=1}+\#V(C)\\
&=
J^-(C)+\#V(C)\\
&=
J^+(C).
\tag{\todaychange{7'}}
\end{align*}
\purplechange{Here the second equality uses Equation~(4), and the last uses
the standard relation \(J^+(C)=J^-(C)+\#V(C)\)
\cite{Arnold1994Book}.  This proves the second assertion.}
\end{proof}

\label{proof:edge}
\begin{proof}[Homology calculation and proof of
Proposition~\ref{prop:edge-curvature}]
The edge complex splits as a direct sum over the double points.  At a fixed
\(v\), it is
\[
0\longrightarrow
\mathbb Z\langle e_l^v,e_r^v\rangle
\xrightarrow{\ \partial\ }
\mathbb Z\langle v\rangle
\longrightarrow0.
\]
An element \(ae_l^v+be_r^v\) is a cycle exactly when
\[
\partial(ae_l^v+be_r^v)=(a+b)v=0,
\]
that is, when \(b=-a\).  Thus
\[
\ker\partial=\mathbb Z\langle e_l^v-e_r^v\rangle.
\]
Because \(v=\partial e_l^v\), the map is surjective and \(H_0=0\).
Taking the direct sum over \(v\) proves
\[
H_1
\cong
\bigoplus_{v\in V(C)}
\mathbb Z\langle[e_l^v-e_r^v]\rangle.
\]

The two terms of the canonical cycle have averaged indices
\(j(v)+\frac12\) and \(j(v)-\frac12\), and turning contributions
\(+\beta_v\) and \(-\beta_v\).  Their normalized weighted contribution is
\begin{align*}
\frac1{2\pi}
\left(
\beta_vq^{j(v)+1/2}
-\beta_vq^{j(v)-1/2}
\right)
&=
\frac{\beta_v}{2\pi}
\left(q^{j(v)+1/2}-q^{j(v)-1/2}\right).
\end{align*}
Summing over all double points proves the total formula.
\end{proof}

\subsection{The canonical two-state homology}
\label{proof:two-state}

\begin{proof}[Proof of Proposition~\ref{prop:two-state}]
Orientation reversal gives, circle by circle,
\[
\ind_{-o}(K)=-\ind_o(K),
\qquad
\rot_{-o}(K)=-\rot_o(K),
\]
and consequently
\[
\varepsilon_{-o}(K)
=
\frac{1-\rot_o(K)}2
=
1-\varepsilon_o(K).
\]
\purplechange{Identifying the same underlying topological circle by the
identity map, its oriented fundamental classes satisfy
\([K]_{-o}=-[K]_o\).  Thus orientation reversal induces}
\[
(1,j,\varepsilon)\longmapsto(1,-j,1-\varepsilon).
\]
Applying orientation reversal twice gives the identity, while changing the
temporary orientation exchanges the two direct-sum factors.

{\color{thisbrown}
\begin{equation*}
\operatorname{grank}_{t,q,u}\mathcal H_{\mathrm{circle}}^{\mathrm{un}}(C)
=\operatorname{grank}_{t,q,u}\mathcal H_{\mathrm{circle}}(C,o)
+\operatorname{grank}_{t,q,u}\mathcal H_{\mathrm{circle}}(C,-o).
\tag{11}
\end{equation*}
}

Put
\[
\purplechange{G_o(q)}
:=
\operatorname{grank}_{t,q,u}\mathcal H_{\mathrm{circle}}(C,o)
\big|_{t=u=-1}
=
I_q(\widetilde C_o)
\]
and
\[
G_{\mathrm{un}}(q)
:=
\operatorname{grank}_{t,q,u}
\mathcal H_{\mathrm{circle}}^{\mathrm{un}}(C)
\big|_{t=u=-1}.
\]
\purplechange{By the preceding degree transformation and
Equation~\eqref{eq:two-state-circle-sum},}
\[
\purplechange{G_{\mathrm{un}}(q)=G_o(q)-G_o(q^{-1}).}
\]
\purplechange{Write \(G_o(q)=\sum_ja_jq^j\).  The three lowest positive
orders are}
\begin{align*}
m=1:\quad
\purplechange{\left.D_q\bigl(G_o(q)-G_o(q^{-1})\bigr)\right|_{q=1}}
&=
\sum_ja_j\bigl(j-(-j)\bigr)
=
\purplechange{2(D_qG_o)(1)},\\
m=2:\quad
\purplechange{\left.D_q^2\bigl(G_o(q)-G_o(q^{-1})\bigr)\right|_{q=1}}
&=
\sum_ja_j\bigl(j^2-(-j)^2\bigr)
=
0,\\
m=3:\quad
\purplechange{\left.D_q^3\bigl(G_o(q)-G_o(q^{-1})\bigr)\right|_{q=1}}
&=
\sum_ja_j\bigl(j^3-(-j)^3\bigr)
=
\purplechange{2(D_q^3G_o)(1)}.
\end{align*}
Generally,
\[
\purplechange{D_q^mG_o(q^{-1})}
=
\sum_ja_j(-j)^mq^{-j},
\]
and hence
{\color{lastbrown}
\begin{equation}
\left.D_q^mG_{\mathrm{un}}(q)\right|_{q=1}
=
\bigl(1-(-1)^m\bigr)
\left.D_q^mG_o(q)\right|_{q=1}
\qquad\purplechange{\because\ \mathrm{(11)}}.
\label{eq:two-state-moments}
\end{equation}
}
\purplechange{Because \(D_q=q\,d/dq\), its value at \(q=1\) equals the
ordinary derivative there.  Hence the case \(m=1\) gives
\(G_{\mathrm{un}}'(1)=2G_o'(1)\).}
Corollary~\ref{cor:circle-Jminus} therefore yields
{\color{lastbrown}
\begin{equation}
J^-(C)=1-G_{\mathrm{un}}'(1)
\qquad\thischange{\because\ \mathrm{(19)}_{m=1},\mathrm{(4)}},
\label{eq:two-state-Jminus}
\end{equation}
}
which is the asserted two-state formula.  Using the classical relation
\(J^+(C)=J^-(C)+\#V(C)\) \purplechange{\cite{Arnold1994Book}} then gives
{\color{lastbrown}
\begin{equation}
J^+(C)=1-G_{\mathrm{un}}'(1)+\#V(C)
\qquad\thischange{\because\ J^+=J^-+\#V(C)},
\label{eq:two-state-Jplus}
\end{equation}
}
which is the stated \(J^+\)-formula.

Finally, Lemma~\ref{lem:smoothed-configuration-invariance} and
Proposition~\ref{prop:component-refinement-invariance} give the same
underlying circle bijection for \(o\) and \(-o\).  It commutes with the
generatorwise orientation reversal, and therefore proves invariance
together with the exchange involution.
\end{proof}

\subsection{The rooted branch decomposition}
\label{proof:branch-invariance}

\begin{proof}[Proof of Proposition~\ref{prop:branch-invariance}]
\branchtext{The ambient isotopy \(F_t\) is the identity outside a compact
disk and therefore carries the unbounded region to the unbounded region.
It induces a graph isomorphism
\[
T(\widetilde C_o)\longrightarrow T(\widetilde C'_{o'})
\]
which fixes the root vertex and sends the edge of \(K\) to the edge of
\(F_1(K)\).  Removing the root commutes with this rooted-tree
isomorphism, so it bijects the actual branches.  The defining condition for
\finalchange{\(\beta_o(K)\)} is that the non-root endpoint of the dual edge of
\(K\) lie in that branch.  The isomorphism therefore bijects the circles assigned to
corresponding branches and preserves all three circle degrees.  \thischange{This proves
invariance of the indexed family}
{\color{thisbrown}
\begin{equation*}
\left(\mathcal H_{\mathrm{circle}}^{i,j,\varepsilon}(C,o;B)\right)_{
B\in\pi_0(T(\widetilde C_o)\setminus\{v_\infty\})}.
\tag{13}
\end{equation*}
}
It also preserves the number of circles in every branch and hence
{\color{thisbrown}
\begin{equation*}
\resizebox{0.98\textwidth}{!}{$\displaystyle
\mathcal P_{\mathrm{branch}}(C,o;t,q,u,w):=
\sum_{B\in\pi_0(T(\widetilde C_o)\setminus\{v_\infty\})}
\left(\sum_{\substack{K\in\pi_0(\widetilde C_o)\\\beta_o(K)=B}}
tq^{\ind_o(K)}u^{\varepsilon(K)}\right)w^{\#\beta_o^{-1}(B)}.$}
\tag{14}
\end{equation*}
}}
\end{proof}

\subsection{The homological face-state-sum}
\label{proof:face-state-sum}

\begin{proof}[Proof of Proposition~\ref{prop:face-state-sum}]
\facetext{For a fixed orientation \(s\), the normalized Alexander coloring
\(R\mapsto\ind_s(R)\) is fixed.  Hence the states are in bijection with
the choices of an actual marked face.  In the product
\(\prod_RW_{R_*}(R)\), all factors
except that of the marked face equal \(1\).  Therefore
\begin{align*}
\sum_{R_*\in\pi_0(\mathbb R^2\setminus\widetilde C_s)}
\prod_RW_{R_*}(R)
&=
\sum_{R_*}
\left(
\sum_i\operatorname{rank}H_i(R_*;\mathbb Z)t^i
\right)
q^{\ind_s(R_*)}[R_*]\\
&=
\sum_j\sum_{R\in\pi_0(R_j(C,s))}\sum_i
\operatorname{rank}H_i(R;\mathbb Z)t^iq^j[R]\purplechange{.}
\tag{\thischange{15}}
\end{align*}
\purplechange{The homomorphism \([R]\mapsto1\) sends Equation~(15) to
\(\operatorname{grank}_{t,q}\mathcal H_{\mathrm{region}}(C,s)\).  The
further specialization \(t=-1\) replaces each face homology by
\(\chi(R)\); Equation~\((\mathrm{I1})\) then gives \(P_C(q)\) for
\(s=o\) and \(P_C(q^{-1})\) for \(s=-o\).}

{\color{thisbrown}
\purplechange{For \(s=o\), applying \([R]\mapsto v^{k(R)}\) to
Equation~(15) gives}
\begin{equation*}
\resizebox{0.98\textwidth}{!}{$\displaystyle
\mathcal P_{\mathrm{region}}(C,o;t,q,v):=
\sum_j\sum_{R\in\pi_0(R_j(C,o))}\sum_i
\operatorname{rank}H_i(R;\mathbb Z)t^iq^jv^{k(R)}.$}
\tag{12}
\end{equation*}
\purplechange{Summing Equation~(15) over the two orientation states gives}
\begin{equation*}
\resizebox{0.98\textwidth}{!}{$\displaystyle
\mathcal Z_{\mathrm{face}}^{\mathrm{un}}(C;t,q):=
\sum_{s\in\{o,-o\}}\sum_j\sum_{R\in\pi_0(R_j(C,s))}\sum_i
\operatorname{rank}H_i(R;\mathbb Z)t^iq^j[s,R].$}
\tag{16}
\end{equation*}
\purplechange{The homomorphism \([s,R]\mapsto1\) sends Equation~(16) to
the graded rank of the direct sum}
\begin{equation*}
\resizebox{0.98\textwidth}{!}{$\displaystyle
\mathcal H_{\mathrm{region}}^{\mathrm{un},i,j}(C):=
\bigoplus_{s\in\{o,-o\}}\mathcal H_{\mathrm{region}}^{i,j}(C,s)
=\mathcal H_{\mathrm{region}}^{i,j}(C,o)\oplus
\mathcal H_{\mathrm{region}}^{i,j}(C,-o).$}
\tag{17}
\end{equation*}
\purplechange{The homomorphism \([s,R]\mapsto v^{k(R)}\) sends
Equation~(16) to its component evaluation}
\begin{equation*}
\mathcal P_{\mathrm{region}}^{\mathrm{un}}(C;t,q,v):=
\sum_{s\in\{o,-o\}}\mathcal P_{\mathrm{region}}(C,s;t,q,v).
\tag{18}
\end{equation*}
}

The ambient isotopy in
Lemma~\ref{lem:smoothed-configuration-invariance} preserves the colored
dual tree, the actual faces, and their homology; hence it identifies the
state sets and every local weight.  This proves invariance.  In the
unoriented expression, changing the temporary orientation only exchanges
the two summands indexed by \(o\) and \(-o\), proving independence of that
choice.}
\end{proof}

{\color{strictgreen}
\subsection{The infinite family and strictness}
\label{proof:strictness}

\begin{proof}[Proof of Proposition~\ref{prop:infinite-pairs} and the
calculation in Proposition~\ref{prop:component-Poincare}]
The repeated blocks in Figure~\ref{fig:infinite-pairs} give
\[
\#V(C_n)=\#V(C'_n).
\]
The complete smoothing configurations for \(n=1,2\), which display the
pattern for general \(n\), are shown in Figure~\ref{fig:smoothing-family}.

\begin{figure}[ht]
\centering
\includegraphics[width=\textwidth]{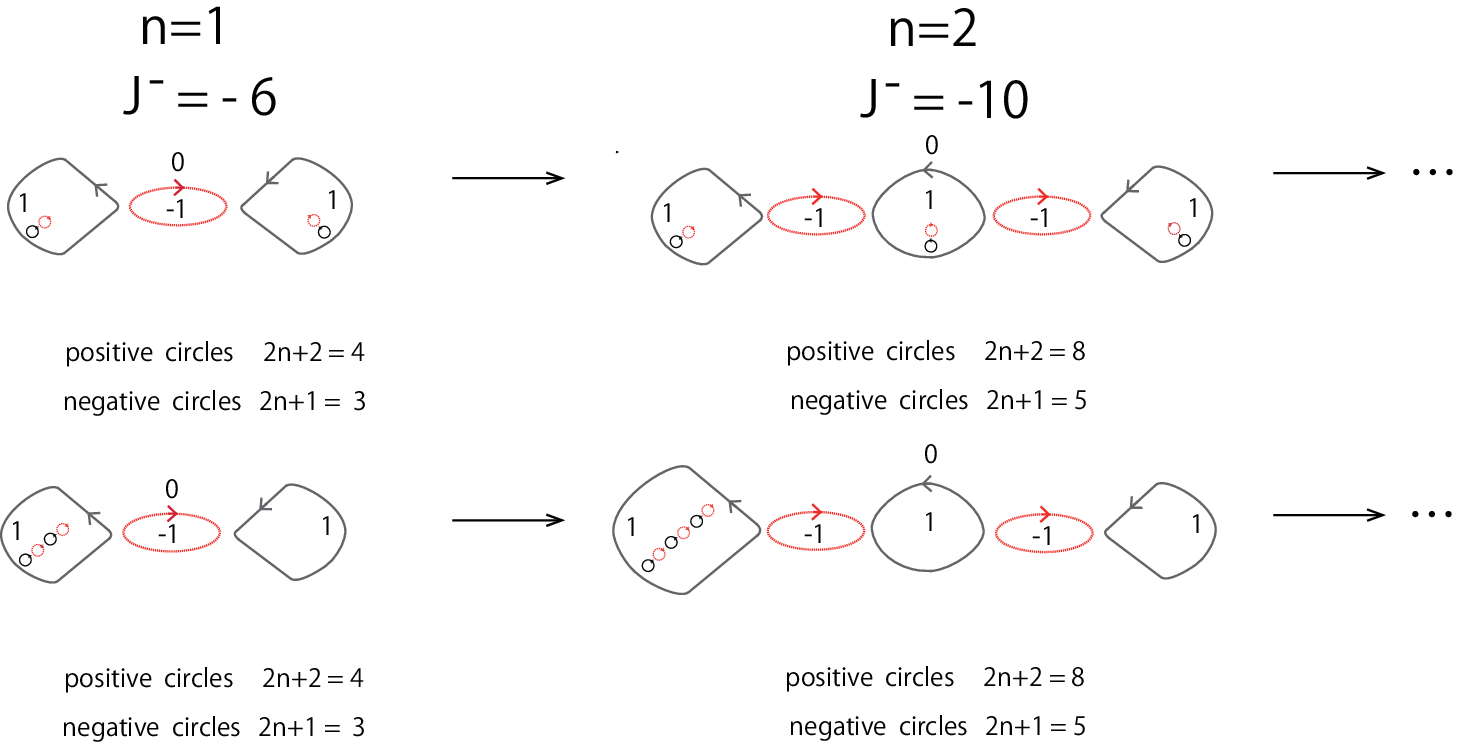}
\caption{\newchange{The smoothing configurations for the first
two pairs.  The orientations, Alexander indices, and positive/negative
circle counts are displayed.}}
\label{fig:smoothing-family}
\end{figure}

The following table records every circle type, its rotation sign, its
averaged Alexander index, and its multiplicity in either member of the
pair:
\[
\begin{array}{c|c|c|c|c}
\text{circle type}&\rot&\ind&C_n&C'_n\\ \hline
\text{outer boundary of an index-one region}
&+1&\tfrac12&n+1&n+1\\
\text{inner positive circle}
&+1&\tfrac32&n+1&n+1\\
\text{inner negative circle}
&-1&\tfrac12&n+1&n+1\\
\text{outer boundary of an index-minus-one region}
&-1&-\tfrac12&n&n
\end{array}
\]
Thus each smoothing has \(2n+2\) positive circles and \(2n+1\) negative
circles.  Equivalently, the complete table of nonzero ordinary triply
graded groups is
\[
\begin{array}{c|c|c}
(i,j,\varepsilon)&
\mathcal H_{\mathrm{circle}}^{i,j,\varepsilon}(C_n,o_n)&
\mathcal H_{\mathrm{circle}}^{i,j,\varepsilon}(C'_n,o'_n)\\ \hline
(1,\tfrac12,1)&\mathbb Z^{n+1}&\mathbb Z^{n+1}\\
(1,\tfrac32,1)&\mathbb Z^{n+1}&\mathbb Z^{n+1}\\
(1,\tfrac12,0)&\mathbb Z^{n+1}&\mathbb Z^{n+1}\\
(1,-\tfrac12,0)&\mathbb Z^n&\mathbb Z^n
\end{array}
\]
with all other groups zero.  Hence the ordinary triply graded
smoothing-circle homologies are isomorphic.

By \((\mathrm{I4})\), their common smoothing polynomial is
\begin{align*}
I_q(\widetilde{(C_n)}_{o_n})
=
I_q(\widetilde{(C'_n)}_{o'_n})
&=
(n+1)q^{1/2}+(n+1)q^{3/2}\\
&\quad{}-(n+1)q^{1/2}-nq^{-1/2}\\
&=
(n+1)q^{3/2}-nq^{-1/2}.
\end{align*}
Equation~\eqref{eq:main-comparison-classical} therefore gives
\(P_{C_n}(q)=P_{C'_n}(q)\).  Moreover,
\begin{align*}
J^-(C_n)=J^-(C'_n)
&=
1-2\left.
\frac{d}{dq}
\left((n+1)q^{3/2}-nq^{-1/2}\right)
\right|_{q=1}
\qquad\purplechange{\because\ \mathrm{(4)}}\\
&=
1-2\left(\frac32(n+1)+\frac12n\right)\\
&=
-4n-2.
\end{align*}
Thus \(n=1\) gives \(-6\) and \(n=2\) gives \(-10\).

At index \(1\), the actual region components are
\[
\begin{array}{c|c|c|c|c|c}
&\text{region type}&\text{number}&H_0&H_1&\latestchange{k(R)}\\ \hline
C_n&\text{disk with two holes}
&n+1&\mathbb Z&\mathbb Z^2&2\\
C'_n&\text{disk with \(2n+2\) holes}
&1&\mathbb Z&\mathbb Z^{2n+2}&2n+2\\
C'_n&\text{disk}
&n&\mathbb Z&0&0
\end{array}
\]
while the data outside index \(1\) agree.  After forgetting the actual
components and the \(k\)-grading, both sides have
\[
H_0(R_1;\mathbb Z)\cong\mathbb Z^{n+1},
\qquad
H_1(R_1;\mathbb Z)\cong\mathbb Z^{2n+2}.
\]
Thus the ordinary bigraded region homologies agree.  In contrast,
\[
\mathcal H_{\mathrm{region}}^{1,1,2}(C_n,o_n)
\cong\mathbb Z^{2n+2},
\qquad
\mathcal H_{\mathrm{region}}^{1,1,2}(C'_n,o'_n)=0,
\]
and consequently
\[
\mathcal H_{\mathrm{region}}^{*,*,*}(C_n,o_n)
\not\cong
\mathcal H_{\mathrm{region}}^{*,*,*}(C'_n,o'_n).
\]

{\color{revisionpurple}
For completeness, the component Poincar\'e evaluation is now calculated
immediately after the region homology.  Each of the \(n+1\) index-one regions
of \(C_n\) contributes
\[
(1+2t)v^2,
\]
so the coefficient of \(q^1\) is
\begin{equation}
(n+1)(1+2t)v^2
\qquad\because\ \mathrm{(12)}.
\label{eq:family-component-Cn}
\end{equation}
For \(C'_n\), the disk with \(2n+2\) holes contributes
\[
\bigl(1+(2n+2)t\bigr)v^{2n+2},
\]
and its \(n\) disk components contribute \(n\).  Thus its coefficient is
\begin{equation}
\bigl(1+(2n+2)t\bigr)v^{2n+2}+n
\qquad\because\ \mathrm{(12)}.
\label{eq:family-component-Cnprime}
\end{equation}
These expressions are unequal.  At \(v=1\), however, both become
\[
n+1+(2n+2)t,
\]
and at \(t=-1\) both give the same Euler-characteristic coefficient
\(-(n+1)\).  The Euler derivative of their difference is
\begin{align*}
&\left.
\left(v\frac{\partial}{\partial v}\right)
\left[
\bigl(1+(2n+2)t\bigr)v^{2n+2}+n
-(n+1)(1+2t)v^2
\right]
\right|_{v=1}\\
&=
(2n+2)\bigl(1+(2n+2)t\bigr)
-2(n+1)(1+2t)\\
&=
4n(n+1)t.
\end{align*}
}

{\color{revisionblue}
We now use the weaker branch structure
\eqref{eq:branch-circle-decomposition}, rather than the full collection of
all relations \(K\subset\partial R\).  For \(C_n\), removing the root of
the dual tree gives \(n+1\) branches with three circle edges and \(n\)
branches with one circle edge.  For \(C'_n\), it gives one branch with
\(2n+3\) circle edges and \(2n\) branches with one circle edge.  Hence the
branch-size multisets are
\[
\{\underbrace{3,\ldots,3}_{n+1},
\underbrace{1,\ldots,1}_{n}\}
\quad\text{and}\quad
\{2n+3,\underbrace{1,\ldots,1}_{2n}\},
\]
which are unequal for every \(n\ge1\).  Therefore
\[
\left(\mathcal H_{\mathrm{circle}}(C_n,o_n;B)\right)_B
\not\cong
\left(\mathcal H_{\mathrm{circle}}(C'_n,o'_n;B')\right)_{B'}
\]
as homologies equipped with their branch decompositions.

The numerical branch evaluations make the same difference explicit.
\purplechange{By Equation~\textup{(14)},}
{\color{lastbrown}
\begin{align}
\mathcal P_{\mathrm{branch}}(C_n,o_n;t,q,u,w)
&=
(n+1)t
\left(q^{1/2}u+q^{3/2}u+q^{1/2}\right)w^3
+ntq^{-1/2}w,
\label{eq:family-branch-Cn}\\
\mathcal P_{\mathrm{branch}}(C'_n,o'_n;t,q,u,w)
&=
t\left(
q^{1/2}u+(n+1)q^{3/2}u+(n+1)q^{1/2}
\right)w^{2n+3}\notag\\
&\quad{}+
nt\left(q^{1/2}u+q^{-1/2}\right)w.
\label{eq:family-branch-Cnprime}
\end{align}
}
They become the same ordinary triply graded rank after \(w=1\), exactly as
the circle table requires.
}

We have therefore exhibited an infinite family on which Viro's polynomial,
its natural ordinary bigraded homological lift, and the ordinary triply
graded circle homology all agree.  The component-graded region homology and
the branch-decomposed smoothing-circle homology distinguish every pair.
This proves the strictness assertion in the stated positive form.
\end{proof}
}

\begin{remark}
\rthreechange{The construction underlying Proposition~\ref{prop:infinite-pairs} is not tied to the particular rooted dual trees displayed above.  The same mechanism admits many further variants obtained by changing the rooted dual-tree pattern while keeping the ordinary graded region and smoothing-circle homologies fixed and altering the component or branch decomposition.  Proposition~\ref{prop:infinite-pairs} records one explicit infinite family of this broader construction.}
\end{remark}

\section*{Acknowledgments and AI tool disclosure}

This work was supported by JSPS KAKENHI Grant Number JP25K06999.
ChatGPT (OpenAI) was used to assist with English-language editing,
expository organization, routine typesetting, and checks of mathematical
notation and internal consistency.  All mathematical statements, proofs,
references, and the final text were independently checked by the author,
who takes full responsibility for the manuscript.

\bibliographystyle{plain}
\bibliography{CatJ}

\end{document}